\documentclass[a4paper,12pt]{article} %{article}
\title{A N\'eron model of the universal jacobian}
\usepackage{amsmath, amssymb, mathrsfs, amsthm, geometry}%, verbatim}% amsfonts, url,, tikz, wrapfig, newclude, }
\usepackage{pgf,tikz}
\usetikzlibrary{arrows}
\usepackage{hyperref}
\usepackage[all]{xy}
\usepackage[OT2, T1]{fontenc}

\usepackage{pgf,tikz}
\usetikzlibrary{arrows}

\usepackage{enumitem}

\usepackage{cleveref}

\newcommand{\on}[1]{\operatorname{#1}}
\newcommand{\bb}[1]{{\mathbb{#1}}}
\newcommand{\cl}[1]{{\mathscr{#1}}}
\newcommand{\ca}[1]{{\mathcal{#1}}}

\newcommand{\ra}{\rightarrow}

\newcommand{\hra}{\hookrightarrow}
\newcommand{\sub}{\subseteq}

\theoremstyle{definition}
\newtheorem{definition}{Definition}[section]

\newtheorem{setup}[definition]{Setup}

\theoremstyle{plain}% default
\newtheorem{proposition}[definition]{Proposition}
\newtheorem{lemma}[definition]{Lemma}
\newtheorem{theorem}[definition]{Theorem}
\newtheorem{corollary}[definition]{Corollary}

\theoremstyle{remark}

\newtheorem{remark}[definition]{Remark}

\newtheorem{example}[definition]{Example}

\def\defeq{\stackrel{\text{\tiny def}}{=}}

\renewcommand{\phi}{\varphi}

\author{David Holmes}
\date{\today}

\newcounter{nootje}
\renewcommand\check[1]{[*\thenootje]\marginpar{\tiny\begin{minipage}
{20mm}\begin{flushleft}\thenootje : 
#1\end{flushleft}\end{minipage}}\addtocounter{nootje}{1}}
\newcommand{\beq}{\begin{equation}}
\newcommand{\eeq}{\end{equation}}
\newcommand{\beqs}{\begin{equation*}}
\newcommand{\eeqs}{\end{equation*}}

\begin{document}
\newcommand{\M}[1]{\ca{M}_{#1}}
\newcommand{\Mbar}[1]{\overline{\ca{M}}_{#1}}
\newcommand{\Mtil}[1]{\widetilde{\ca{M}}_{#1}}
\newcommand{\MN}[1]{{\ca{N}}_{#1}}

\maketitle
\begin{abstract} 
The jacobian of the universal curve over $\M{g,n}$ is an abelian scheme over $\M{g,n}$. Our main result is the construction of an algebraic space $\beta\colon \Mtil{g,n} \ra \Mbar{g,n}$ over which this jacobian admits a N\'eron model, and such that $\beta$ is universal with respect to this property. We prove certain basic properties, for example that $\beta$ is separated, locally of finite presentation, and satisfies a certain restricted form of the valuative criterion for properness. In general, $\beta$ is not quasi-compact. We relate our construction to Caporaso's balanced Picard stack $\ca{P}_{d,g}$. % and to Maino's moduli of enriched stable curves. 
%In general this `universal jacobian' does not admit a N\'eron model over $\Mbar{g,n}$, even after blowing up or altering the base. We construct a `universal' morphism $\Mtil{g,n} \ra \Mbar{g,n}$ over which the universal jacobian does admit a (finite type) N\'eron model, and we prove certain basic properties. We relate our construction to Caporaso's balanced Picard stack $\ca{P}_{d,g}$. 
%
%Can we give an application?!
%
%Can we say anything more about the resulting map (I hope it exists) to compactified moduli of abelian varieties?
%
%At the moment we have a N\'eron model, and therefore an AJ map trivially. However, we don't seem to say anything about maps to the moduli space of abelian varieties, which seems to be relevant for the potential applications. 
%
%Compactified jacobians of aligned curves?!
%
%Does the push forward map on cycles form the universal blowup to the base respect rational equivalence? CF. Burgos paper on b-divisors. I guess that b-Weil divisors correspond to locally finite sums of prime divisors. Pushforward on these certainly can't make sense in general, but we want to define a nice sub-class, closed under intersections, and on which the push forward is defined. 
%
%Can we show our object somehow `dominates every mumford-lear extension' in the sense of Jose et al? Perhaps:Given a mumford-lear extension to some $S'\ ra S$, there exists $S' \ra S'' \ra S$ such that that $\beta\colon \tilde{S} \ra S$ factors via $S'' \ra S$ and such that the M-L extension on $S'$ comes via pull-back from one on $S''$. 
%
\end{abstract}

\tableofcontents

\newcommand{\et}{^{\on{et}}}

\section{Introduction}
We write $\M{g,n}$ for the Deligne-Mumford stack of smooth proper curves of genus $g$ with $n$ disjoint ordered marked sections, and $\Mbar{g,n}$ for its Deligne-Knudsen-Mumford compactification. The universal curve over $\M{g,n}$ is smooth and proper, and hence has a jacobian, an abelian scheme, which we shall denote $J_{g,n}/\M{g,n}$. Numerous authors have investigated (partial) compactifications of $J_{g,n}$ over $\Mbar{g,n}$, see for example \cite{Caporaso2008Neron-models-an}, \cite{Chiodo2015Neron-models-of}, \cite{Esteves2001Compactifying-t}, \cite{Frenkel2009Gromov-Witten-g}, \cite{Melo2009Compactified-Pi}. 

N\'eron models \cite{Neron1964Modeles-minimau} are the best partial compactifications for families of abelian varieties over 1-dimensional bases. For higher-dimensional families the definition of a N\'eron model still makes sense (see \cref{def:nm}), but in general they do not exist \cite{Holmes2014Neron-models-an}. The main result of this paper is the construction of a certain universal base-change $\Mtil{g,n} \ra \Mbar{g,n}$ after which a N\'eron model \emph{does} exist. 

More precisely, we construct a map $\beta\colon \Mtil{g,n} \ra \Mbar{g,n}$ which is an isomorphism over $\M{g,n}$ and such that
\begin{enumerate}
\item[(1)]
 the universal jacobian $J_{g,n}$ admits a N\'eron model over $\Mtil{g,n}$;
 \item[(2)] $\Mtil{g,n}\ra \Mbar{g,n}$ is universal with respect to property (1) --- see \cref{def:UNMAM} for a precise statement.
\end{enumerate}
We show that $\beta$ is separated, locally of finite type, relatively representably by algebraic spaces, and that it satisfies a restricted form of the valuative criterion for properness (see \cref{theorem:intro_main}). In general, $\beta$ is not quasi-compact. 

The map $\beta$ is an isomorphism exactly over the locus of `treelike curves' (curves whose dual graph has the property that every circuit in the graph has only one edge). This locus contains the curves of compact type, and also the locus denoted $\Mbar{g,n}^\circ$ in the work of Grushevsky and Zakharov \cite{Grushevsky2012The-zero-sectio} (and also used by Dudin \cite{Dudin2015Compactified-un}) for the study of the double ramification cycle. Indeed, a N\'eron model tautologically admits an abel-jacobi map, so it is perhaps unsurprising that the double ramification cycle is more accessible on this locus. 

%The algebraic space $\Mtil{g,n}/\Mbar{g,n}$ is defined by a universal property, but from this description alone it is not clear what the functor of points of $\Mtil{g,n}$ is on an arbitrary test object. It turns out that this can be described by a variation of Main\`o's `enriched structures' (cf. \cite{Maino1998Moduli-space-of}, \cite{Pacini2009Enriched-spin-c}, \cite{Osserman2014Dimension-count}, \cite{Osserman2014Limit-linear-se}); more precisely, that $\Mtil{g,n}/\Mbar{g,n}$ is a moduli space for generalised enriched structures. This will be described in a subsequent paper.  

\subsection{Statements of main results}

%
%Say exactly where an iso (give it a name? X for now)
%
%talk about AJ map:
%- because we have a NM, we get tautologically a map from the smooth locus of the universal curve. 
%- The NM does not have connected fibres, in particular not semi-abelian. Hence don't get an extension of AJ in sense of map from Mg,n to a shimura variety/`linear' type moduli space.
%- over X, do get a map because component groups trivial
%in general, if choose the points suitably, end up in connected component of identity, so again get god AJ map
%in general, some tower of maps from Mgtilde to universal semi-abelian over second voronoi?! Does the latter exist?
%
%
%
%Mention compact type and also the weird condition of GZ. 
%
%Mention Caporaso, Maino. 
%
%the best models are N\'eron models. In higher dimensions, the definition of NM still makes sense (just apply NMP), but the problem is that N\'eron models do not in general exist, and we must usually content ourselves with partial compactifications with less good properties. The main result of this paper is that, in the case of the universal jacobian, a N\'eron model does exist after a certain base-change. 
%
%More precisely, 
%
%...
%
%...
%
%Be more positive!!!
%
%Put the technical definitions later on!!!
%
%bear in mind that Caporaso might review. 
%...
%
%
%One can ask whether this abelian scheme admits a N\'eron model over $\Mbar{g,n}$. We need a definition:

We begin with the definition of a N\'eron model. This is the natural generalisation of \cite[definition 1.1]{Holmes2014Neron-models-an} to the situation where the base is a stack. %In particular, if $S$ is a scheme then this definition reduces to \cite[definition 1.1]{Holmes2014Neron-models-an}. 
\begin{definition}\label{def:nm}
Let $S$ be a reduced algebraic stack and $U\hra S$ a dense open substack. Let $A/U$ be an abelian scheme. A \emph{N\'eron model for $A$ over $S$} is a smooth separated group algebraic space $N/S$ together with an isomorphism $A \ra N\times_S U$ satisfying the following universal property:
\emph{let $T \ra S$ be a smooth morphism representable by algebraic spaces, and $f\colon T_U \ra A = N_U$ any $U$-morphism. Then there exists a unique $S$-morphism $F\colon T \ra N$ such that $F|_{T_U} = f$.} %, and moreover this morphism $F$ is unique up to unique isomorphism.}
\end{definition}

By the results of \cite{Holmes2014Neron-models-an}, it is clear that the universal jacobian over $\Mbar{g,n}$ does not in general admit a N\'eron model - more precisely, a N\'eron model will exist if and only if $g=n=1$ or $g=0$. Further, the same reference shows that blowing up or altering $\Mbar{g,n}$ will not affect this non-existence. On the other hand, we deduce from \cite[corollary 1.3]{Holmes2014Neron-models-an} that there is an open substack $U$ of $\Mbar{g,n}$ whose complement has codimension 2 and over which a (finite type) N\'eron model exists, and we also know that if $T$ is any trait\footnote{The spectrum of a discrete valuation ring. } and $f\colon T \ra \Mbar{g,n}$ is any map sending the generic point of $T$ to a point in $\M{g,n}$, then $f^*J_{g,n}$ admits a (finite type) N\'eron model. 

Motivated by this, we ask for the `best' morphism to $\Mbar{g,n}$ such that the pullback of the universal Jacobian $J_{g,n}$ admits a N\'eron model. To make this precise, we use a universal property:

\begin{definition}\label{def:UNMAM}
A \emph{N\'eron-model-admitting morphism} is a morphism $f\colon T \ra \Mbar{g,n}$ of algebraic stacks such that:
\begin{enumerate}
\item $T$ is regular;
\item $f^{-1}\M{g,n}$ is dense in $T$;
\item $f^*J_{g,n}$ admits a N\'eron model over $T$. 
\end{enumerate}

A \emph{universal N\'eron-model-admitting morphism} is a terminal object in the 2-category of N\'eron-model-admitting morphisms. 
\end{definition}

Our main result is the following:
\begin{theorem}[\cref{thm:main}]\label{theorem:intro_main}
A universal N\'eron-model-admitting morphism $\beta\colon \Mtil{g,n} \ra \Mbar{g,n}$ exists. 
\end{theorem}
The morphism $\beta$ is separated, locally of finite presentation, and relatively representable by algebraic spaces. Moreover, $\beta$ satisfies the valuative criterion for properness for morphisms from traits to $\Mbar{g,n}$ which map the generic point of the trait to a point in $\M{g,n}$. The N\'eron model of $J_{g,n}$ over $\Mtil{g,n}$ is of finite type over $\Mtil{g,n}$, and its fibrewise-connected-component-of-identity is semiabelian. The stack $\Mtil{g,h}$ is smooth over $\bb{Z}$. 

Note that $\beta$ is in general not quasi-compact (in fact, it is quasi-compact if and only if it is an isomorphism, if and only if $g=n=1$ or $g=0$). The statement about the valuative criterion for properness for certain traits can be thought of as saying `suitable test curves in $\M{g,n}$ extend to $\Mtil{g,n}$'. One might reasonably ask whether the same result holds true if $\Mbar{g,n+1}/\Mbar{g,n}$ is replaced by some other family of nodal or stable curves. In general the answer seems to be no; we make essential use of the precise structure of the non-smooth locus of $\Mbar{g,n+1}/\Mbar{g,n}$. The most general situation to which our results apply is described in \cref{thm:NM_over_univ_aligning_over_scheme}.

\subsection{Comparison to Caporaso's balanced Picard stack}
Let $g \ge 3$, and let $d$ be an integer such that $\on{gcd}(d-g+1,2g-2)=1$. In \cite{Caporaso2008Neron-models-an}, Caporaso constructs a `balanced Picard stack' $\ca{P}_{d,g} \ra \Mbar{g}$, and shows that it acts as a `universal N\'eron model' for smooth test curves $f\colon T \ra \Mbar{g}$ such that the corresponding stable curve $X/T$ is regular. More precisely, for any such test curve, the scheme $f^*\ca{P}_{d,g}$ is canonically isomorphic to the N\'eron model of the degree-$d$ jacobian of the generic fibre of $X$. Note that this canonical isomorphism is not a morphism of group spaces or torsors. 

Our N\'eron model $N_{g}$ of the universal jacobian $J_{g}$ satisfies a similar property; it follows easily from the definition that for any test curve $f\colon T \ra \Mbar{g}$ such that the given curve $X/T$ is regular, the map $f$ factors (uniquely up to unique isomorphism) via a map $g\colon T \ra \Mtil{g}$ and moreover we have that $g^*N_{g}$ is canonically isomorphic \emph{as a group space} to the N\'eron model of the jacobian of the generic fibre of $X \ra T$. Moreover, we will see in \cref{sec:caporaso} that we can also treat the degree-$d$ jacobian in a similar way. 

These properties may seem rather similar, but there are important differences:
\begin{enumerate}
\item
 Caporaso's $\ca{P}_{d,g}$ is a scheme over $\Mbar{g}$, whereas our N\'eron model $N_{g}$ can only exist over $\Mtil{g}$;
 \item
  The scheme $\ca{P}_{d,g}$ can never admit a group or torsor structure compatible with that on the jacobian, whereas our N\'eron model $N_{g}$ is automatically a group space; 
\end{enumerate}

The scheme $\ca{P}_{d,g}$ has a good universal property for test curves $f\colon T \ra \Mbar{g}$ such that the pullback of the universal curve to $T$ is regular - roughly speaking, the test curve is `sufficiently transverse to the boundary'. In contrast, our construction has a good universal property for all morphisms from test curves mapping the generic point into $\M{g}$, and even for test object of higher dimensions. Indeed, test curves $f\colon T \ra \Mbar{g}$ such that the pullback of the universal curve to $T$ is regular factor through a quasi-compact open subscheme $\Mtil{g}^{\le 1}$ of $\Mtil{g}$. A precise description of $\Mtil{g}^{\le 1}$ is given in \cref{sec:caporaso}, where we also show that $\ca{P}_{d,g}$ \emph{does} admit a torsor structure after pullback to $\Mtil{g}^{\le 1}$, and give a more precise comparison to our N\'eron model. 

In summary, it seems reasonable to say that our N\'eron model $N_{g}$ over $\Mtil{g}$ has a stronger universal property than Caporaso's $\ca{P}_{d,g}$, but this comes at the price of replacing $\Mbar{g}$ by $\Mtil{g}$.

\subsection{Outline of the paper}

In \cref{sec:definition_of_balance} we give the definition of an aligned curve - this is very slightly different from the definition given in \cite{Holmes2014Neron-models-an} but is equivalent in the situations we care about, see \cref{remark:reg_irreg_aligned}. Recall from \cite[theorem 1.2]{Holmes2014Neron-models-an} that there is a close connection between a curve being aligned and its jacobian admitting a N\'eron model. 

In \cref{sec:univ_aligning_for_curves} we define `universal aligning morphisms'; the proof that they exist is in \cref{sec:scheme_representing_relations}. The sections in-between contain various intermediate results we need for the proof. In \cref{sec:regularity_for_correlation} we show that, if the singularities of $C/S$ are `mild enough', then the universal aligning scheme is in fact regular (in particular, this holds in the universal case). We also show that the construction of the universal aligning morphism can be made slightly more explicit in this situation. 

In \cref{sec:resolving_sings_of_pullback} we will show that, again if $C/S$ has mild singularities, it is possible to resolve the singularities of the pulled-back curve over the universal aligning scheme. In \cref{sec:NM_over_univ_aligning} we will apply this so show that the universal aligning morphism is in fact a universal N\'eron model admitting morphism.

\Cref{sec:example} contains a worked example, and in \cref{sec:caporaso} we relate our results to some constructions of Caporaso from \cite{Caporaso2008Neron-models-an}.

In this paper, `algebraic stack' means `algebraic stack in the sense of \cite[\href{http://stacks.math.columbia.edu/tag/026O}{Tag 026O}]{stacks-project}'. 

\subsection{Acknowledgements}
The author would like to thank Omid Amini, Owen Biesel, Jos\'e Burgos Gil, Alessandro Chiodo, Bas Edixhoven, Eduardo Esteves, Robin de Jong, Bashar Dudin, Margarida Melo, Filippo Viviani and Morihiko Saito for helpful comments and discussions.

%K\k{e}stutis \v{C}esnavi\v{c}ius, , Ariyan Javanpeykar, , Qing Liu and  for helpful comments and discussions. 

\section{Definition and basic properties of aligned curves}\label{sec:definition_of_balance}
\newcommand{\mono}{\textbf{Mon}_0}

In this section we briefly recall the definitions from \cite{Holmes2014Neron-models-an}, where more precise statements can be found. For us `monoid' means `commutative monoid with zero', and a graph has finitely many edges and vertices, and is allowed loops and multiple edges between vertices. We are interested in graphs whose edges are labelled by elements of a monoid.  

%
%\begin{definition}
%We write $\mono$ for the category of commutative monoids with zero, where we require that maps send $0$ to $0$ and $1$ to $1$. We usually write monoids multiplicatively. 
%%A monoid $M$ is \emph{cyclic} if it can be generated by a single element. 
%\end{definition}
%
%\begin{definition}
%A graph is a triple $(V,E,\on{ends})$ where $V$ and $E$ are sets (vertices and edges), and $\on{ends}\colon E \ra (V\times V)/\on{S}_2$ is a function, which we think of as assigning to each edge the unordered pair of its endpoints. 
%
%Let $L$ be a monoid or set, and $\Gamma$ a graph. An \emph{edge-labelling} of $\Gamma$ by $L$ is a function $\ell$ from the set of edges of $\Gamma$ to $L$. We call the pair $(\Gamma, \ell)$ a \emph{graph with edges labelled by $L$}.
%\end{definition}

\begin{definition}
Let $L$ be a monoid, and $(H, \ell)$ a 2-vertex-connected\footnote{i.e. connected and remains connected after removing any one vertex. } graph labelled by $L$ (so $\ell$ is a map from the set of edges of $H$ to $L$). We say $(H, \ell)$ is \emph{aligned} if there exists $l \in L$ and positive integers $n(e)$ for each edge $e$ such that for all edges $e$ we have
$$\ell(e) = l^{n(e)}. $$

Let $L$ be a monoid, and $(G, \ell)$ a graph labelled by $L$. We say $G$ is \emph{aligned} if every 2-vertex connected subgraph of $G$ is aligned. 
\end{definition}

Now we turn to nodal curves, by which we mean proper, flat curves whose geometric fibres are connected and have at worst ordinary double point singularities. These were called `semistable curves' in \cite{Holmes2014Neron-models-an}. We recall from \cite[propositions 2.5 and 2.9]{Holmes2014Neron-models-an} a result on the local structure of such curves:

%Our main result on the local structure of such curves is:
%
%\begin{definition}
%Let $k$ be a separably closed field. A \emph{curve over $k$} is a finitely presented morphism $\pi\colon C \ra \on{Spec} k$ such that every irreducible component of $C/k$ is of dimension $1$. The curve $C/k$ is called \emph{nodal} if $\pi$ is proper and if every point on $C$ is either smooth, or has completed local ring isomorphic to $k[[x,y]]/(xy)$ (i.e. only ordinary double point singularities). 
%
%Let $S$ be a scheme. A \emph{(nodal) curve over $S$} is a flat finitely presented morphism $C \ra S$, all of whose   fibres over points with values in separably closed fields are (nodal) curves. 
%\end{definition}
\begin{proposition}\label{local_rings_on_nodal}
Let $S$ be a locally noetherian scheme, $C/S$ a nodal curve, $s$ a geometric point of $S$, and $c$ a geometric point of $C$ lying over $s$. Then there exists an element $\alpha$ in the maximal ideal of the \'etale local ring $ \ca{O}\et_{S,s}$ and an isomorphism of complete local rings
$$ \frac{ \widehat{{\ca{O}}\et_{S,s}}[[x,y]]}{(xy-\alpha)} \ra \widehat{\ca{O}\et_{C,c}} . $$
The element $\alpha$ is not in general unique, but, the ideal $\alpha \ca{O}\et_{S,s} \triangleleft \ca{O}\et_{S,s}$ is unique. We call it the \emph{singular ideal of $c$}. If $C/S$ is smooth over a scheme-theoretically-dense open of $S$ then the singular ideal is never a zero-divisor. 
\end{proposition}
%
%\begin{proof}
%See propositions 2.5 and 2.9 of \cite{Holmes2014Neron-models-an}. 
%\end{proof}

%\begin{lemma}
%Let $\pi\colon C\ra S$ be nodal with $(S,s)$ strictly henselian. Assume that $C \ra S$ is smooth over some scheme-theoretically dense open subscheme $U \sub S$. Then for all non-smooth points $c \in C$, the singular ideal of $c$ is generated by a non-zero-divisor in $\ca{O}_{S,\pi(c)}$. 
%\end{lemma}

\begin{definition} \label{def:balance}
Let $S$ be a locally noetherian scheme and $C/S$ a nodal curve. Let $s \in S$ be a geometric point, and write $\Gamma$ for the dual graph\footnote{ Defined as in \cite[10.3.17]{Liu2002Algebraic-geome}. } of the fibre $C_s$.  Let $\on{Prin}(\ca{O}\et_{S,s})$ be the monoid of principal ideals of $\ca{O}\et_{S,s}$, so $\on{Prin}(\ca{O}\et_{S,s}) = \ca{O}\et_{S,s}/(\ca{O}\et_{S,s})^\times$. We label $\Gamma$ by elements of $\on{Prin}(\ca{O}\et_{S,s})$ by assigning to an edge $e \in \Gamma$ the singular ideal $l \in \on{Prin}(\ca{O}\et_{S,s})$ of the singular point of $C_s$ associated to $e$. % (cf. \cref{local_rings_on_nodal}). 

We say \emph{$C/S$ is aligned at $s$} if and only if this labelled graph is aligned. We say \emph{$C/S$ is aligned} if it is aligned at $s$ for every geometric point $s \in S$. 
\end{definition}

\begin{remark}\label{remark:reg_irreg_aligned}
In \cite{Holmes2014Neron-models-an} we used a slightly different definition of alignment, namely we said $(H, \ell)$ was aligned if for all pairs of edges $e$, $e'$ there exist positive integers $n$, $n'$ such that $\ell(e)^{n} = \ell(e')^{n'}$. When we want to distinguish between these notions we will call the version from \cite{Holmes2014Neron-models-an} `irregularly aligned' and the new version in this paper `regularly aligned'. If $C/S$ is regularly aligned then it is irregularly aligned, and if $S$ has factorial \'etale local rings (for example if $S$ is regular) then the converse holds. In this paper, we will construct a universal regularly-aligning morphism, and a N\'eron model over it. A universal irregularly-aligning morphism also exists, but the universal object is not in general regular, and so we cannot construct a N\'eron model over it (probably one does not exist). 
\end{remark}

\begin{remark}\label{remark:fppf_descent}
The notion of alignment is fppf-local on the target, i.e. it is preserved under flat base-change and satisfies fppf descent. We sketch a proof of this fact: let $R \ra R'$ be a faithfully flat ring map, and let $r_1$, $r_2 \in R$ and $a$, $u_1$, $u_2 \in R'$ with $u_i$ units and $r_i = u_ia^{n_i}$ for some $n_i \in \bb{Z}_{>0}$. Suppose the map $$R \ra R[t, u_1^{\pm1}, u_2^{\pm 1}]/(r_1 - u_1t^{n_1}, r_2 - u_2t^{n_2})$$ becomes an isomorphism after base-change to $R'$, then by faithful flatness it was an isomorphism. 

\end{remark}

\begin{definition}
Let $S$ be a locally noetherian algebraic stack, and $C/S$ a nodal curve. Let $S' \ra S$ be a smooth cover by a scheme. We say that $C/S$ is aligned if and only if $C\times_S S' \ra S'$ is aligned. This makes sense by \cref{remark:fppf_descent}. 
\end{definition}

\section{Universal aligning morphisms}\label{sec:univ_aligning_for_curves}

\begin{definition}\label{def:univ_reg_aligning_for_curves}
Let $S$ be a locally noetherian and reduced algebraic stack, and $U \hra S$ a dense open substack. Let $C/S$ be a nodal curve which is smooth over $U$. 

\begin{enumerate}
\item Let $f \colon T \ra S$ be a morphism of algebraic stacks. We say that $f$ is \emph{non-degenerate} if $T$ is reduced and  $U \times_S T$ is dense in $T$. 

\item Let $f\colon T \ra S$ be a non-degenerate morphism of locally-noetherian stacks. We say $f$ is \emph{aligning} if $C \times_S T \ra T$ is aligned. 

\item The category of aligning morphisms is defined as the full sub-2-category of stacks over $S$ whose objects are aligning morphisms. 

\item A \emph{universal aligning morphism} is a terminal object in the 2-category of aligning morphisms over $S$. 
\end{enumerate}
\end{definition}
\begin{remark}\leavevmode
\begin{enumerate}
\item A large part of the paper is taken up with proving the existence and certain properties of a universal aligning morphism in the case of the universal stable curve. 
\item We will show that the universal aligning morphism is in fact a separated algebraic space locally of finite type over $S$. 
\end{enumerate}
\end{remark}

Later we will show that universal aligning morphisms exist, but for now we will just show that they pull back nicely if they do exist. 
\begin{lemma}\label{lem:tilde_pullback}
Let $S$ be a reduced locally noetherian stack, $C/S$ a nodal curve smooth over a dense open $U\hra S$, and let $f\colon T \ra S$ be a non-degenerate morphism. Write $\tilde{S}/S$ for the universal aligning morphism for $C/S$, and $\tilde{T}/T$ for the universal aligning morphism for $C_T/T$ - we assume that both exist. Then
\begin{enumerate}
\item
The natural map $f^*\tilde{S} \ra T$ is aligning for $C_T/T$, so we get a map $\phi\colon f^*\tilde{S} \ra \tilde{T}$;
\item The map $\phi$ is an isomorphism. 
\end{enumerate}
\end{lemma}
\begin{proof}\leavevmode
\begin{enumerate}
\item
The pullback of an aligned curve is again aligned, so we immediately get that $f^*\tilde{S} \ra T$ is aligning for $C_T/T$. 
\item The composite $\tilde{f}\colon \tilde{T} \ra T \ra S$ is non-degenerate since $f^{-1}U$ is dense in $\tilde{T}$. Also, $\tilde{f}^*C/\tilde{T}$ is aligned by definition of $\tilde{T}$, so $\tilde{f}$ is aligning. Hence $\tilde{f}$ factors via $\tilde{S}\ra S$, yielding a map $\psi\colon\tilde{T} \ra f^*\tilde{S}$. The uniqueness part of the universal property then shows that $\psi$ is an inverse to $\phi$. 
\end{enumerate}
\end{proof}

\section{Quasisplit curves}
In order to construct the universal aligning morphism by descent, we will make use of the notion of a `quasisplit' nodal curve. This differs from (though is somewhat similar to) the notion of a split nodal curve in \cite{Jong1996Smoothness-semi}; neither is stronger than the other. 

%I think this notion may only be good (or even defined) over schemes - else what is the Zariski local ring?

\begin{definition}\label{definition:quasisplit}
Let $S$ be a scheme, and let $C/S$ be a nodal curve. We write $\on{Sing}(C/S)$ for the closed subscheme of $C$ where $C \ra S$ is not smooth (more precisely, it is the closed subscheme cut out by the first fitting ideal of the sheaf of relative 1-forms of $C/S$). We say $C/S$ is \emph{quasisplit} if the following two conditions hold:
\begin{enumerate}[label={\ref{definition:quasisplit}.\arabic*}]
\item \label{item:source-locally-closed-imm}
the morphism $\on{Sing}(C/S) \ra S$ is an immersion Zariski-locally on the source (for example, a disjoint union of closed immersions);%\chck{I think we want to weaken this to `source-locally an immersion'}
\item \label{item:split_fibres} for every field-valued fibre $C_k$ of $C/S$, every irreducible component of $C_k$ is geometrically irreducible. 
\end{enumerate}
\end{definition}

\begin{remark}\label{remark:no_geom_points_for_quasisplit}
Given a quasisplit nodal curve $C/S$ and a geometric point $\bar{s}$ of $S$ with image $s \in S$, the graph $\Gamma_{\bar{s}}$ depends only on $s$ and not on $\bar{s}$. As such, it makes sense to talk about the dual graph of $C_s$ for $s \in S$ a point in. Moreover, the labels on such a graph, which a-priori live in the \'etale local ring at $\bar{s}$, are easily seen to live in the henselisation of the Zariski local ring, by condition \ref{item:split_fibres}. Applying condition \ref{item:source-locally-closed-imm}, we see that the labels in fact live in the Zariski local ring at $s$. 

Summarising, given a quasisplit curve $C/S$ and a point $s \in S$, it makes sense to talk about the labelled dual graph of the fibre of $C$ over $s$, and the labels live in the Zariski local ring at $s$. For the remainder of this paper, we will do this without further comment.  
\end{remark}

\begin{lemma}\label{lem:quasisplit_after_etale_cover}
Let $S$ be a locally noetherian scheme and $C/S$ a nodal curve. Then there exists an \'etale cover $f\colon S' \ra S$ such that $f^*C \ra S'$ is quasisplit. 
\end{lemma}
\begin{proof}
First we construct an \'etale cover over which the non-smooth locus is a disjoint union of closed immersions. Note first that this holds after base-change to the \'etale local ring at any geometric point, by the structure of finite unramified modules over strictly henselian local rings \cite[\href{http://stacks.math.columbia.edu/tag/04GL}{Tag 04GL}]{stacks-project}. Since $C/S$ is of finite presentation, one can show by a standard argument that the same statement holds on some \'etale cover. 

The non-smooth locus being a union of closed immersions is a property stable under base-change, so we now construct a further \'etale cover over which the irreducible components of fibres become geometrically irreducible. This is easy: if $C\ra S$ admits a section through the smooth locus of a given irreducible component of a field-valued fibre, then that component is geometrically irreducible. Moreover, the smooth locus of $C/S$ is dense in every fibre. As such, it is enough to show that, \'etale-locally, there is a section through the smooth locus of every component of every field-valued fibre. This holds since a smooth morphism admits \'etale-local sections through every point. 
\end{proof}

 \section{Specialisation maps between labelled graphs for quasisplit curves}\label{sec:specialisation_maps}
Given a quasisplit nodal curve $C/S$, and two points $s$, $\zeta$ of $S$ with $s$ in the closure of $\zeta$, we will show that there is a well-defined `specialisation map' $\Gamma_s \ra \Gamma_\zeta$ on the labelled graphs. First, we need to define a morphism of labelled graphs. 

\begin{definition}
%Let $\Gamma = (V,E, \on{ends})$ and $\Gamma' = (V', E', \on{ends}')$ be two graphs. A \emph{morphism of graphs} $\phi\colon \Gamma \ra \Gamma'$ consists of
%\begin{enumerate}
%\item
%a map of sets $\phi_V\colon V \ra V'$;
%\item a map of sets $\phi_E\colon E \ra V' \cup E'$
%\end{enumerate}
%satisfying the following conditions for every $e \in E$:
%\begin{enumerate}
%\item
%if $\phi_E(e) \in E'$, then $(\phi_V\times \phi_V)(\on{ends}(e)) = \on{ends}'(\phi_E(e))$;
%\item if $\phi_E(e) \in V'$ then (writing $\on{ends}(e) =\{v_1, v_2\}$) we have that $\phi_V(v_1) = \phi_V(v_2) = \phi_E(e)$. 
%\end{enumerate}
%In words, vertices of $\Gamma$ are sent to vertices of $\Gamma'$, and edges of $\Gamma$ are either sent to edges of $\Gamma'$ or are contracted to vertices of $\Gamma'$. 
A morphism of graphs sends vertices to vertices, and sends edges to either edges or vertices (thinking of the latter as `contracting an edge'), such that the obvious compatibility conditions hold. 

 A morphism of edge-labelled graphs
 $$\phi\colon (\Gamma, \ell\colon \on{edges}(\Gamma) \ra L) \ra (\Gamma', \ell'\colon \on{edges}(\Gamma') \ra L')$$
is a pair consisting of a morphism of graphs $\Gamma \ra \Gamma'$ and a morphism of monoids from $L$ to $L'$ such that the labellings on non-contracted edges match up. An isomorphism is a morphism with a two-sided inverse. 
\end{definition}

\begin{definition}\label{def:contraction_map}
Let $S$ be a locally noetherian scheme, and let $C/S$ be a quasisplit nodal curve. Let $s$, $\zeta \in S$ be two points such that $\zeta$ specialises to $s$ (i.e. $s \in \overline{\{\zeta\}}$). Write $\Gamma_s$ and $\Gamma_\zeta$ for the corresponding labelled graphs  - recall from \cref{remark:no_geom_points_for_quasisplit} that this makes sense without choosing separable closures of the residue fields, and moreover that the labels of $\Gamma_s$ and $\Gamma_\zeta$ may be taken to lie in the Zariski local rings at $s$ and $\zeta$ respectively. Write 
$$\on{sp}\colon \ca{O}_{S,s} \hra \ca{O}_{S,\zeta}$$
for the canonical (injective) map. 

We will define a map of labelled graphs 
$$\phi\colon \Gamma_s \ra \Gamma_\zeta, $$
writing $\phi_V$ for the map on vertices and $\phi_E$ for the map on edges. First we define the map on vertices. Let $v \in V_\zeta$ be a vertex of $\Gamma_\zeta$. Then $v$ corresponds to an irreducible component $v$ of the fibre over $\zeta$. Let $\cl{V}$ denote the Zariski closure of this component in $C$. Then $\cl{V} \times_S s$ is a union of irreducible components of $C_s$, call them $v_1, \cdots, v_n$. We define $\phi_V(v_i) = v$ for $1 \le i \le n$. One checks easily that each irreducible component of $C_s$ arrises in this way from exactly one vertex of $\Gamma_\zeta$, so we obtain a well-defined map $\phi_V\colon V_s \ra V_\zeta$. 

Next we define the map on edges. Write $Z = \overline{\{\zeta\}} \sub S$ for the closure of $\zeta$. Let $e\in E_s$ be an edge of $\Gamma_s$. Then there are exactly two possibilities\footnote{To see this, note first that a section in $C_T/T$ as in case 1 is unique if it exists. Suppose we have $\phi_V(v_1) \neq \phi_V(v_2)$. Then $v_1$ is contained in some irreducible component $T_1$ of $C_T$, and $v_2$ is in a component $T_2$, with $T_1 \neq T_2$. In this situation observe that $e \in T_1 \cap T_2$, and (by considering the local structure of the singularities of a quasisplit nodal curve) we find that $T_1 \cap T_2$ is locally on $Z$ a union of sections, so case 1 must hold. }:
\begin{enumerate}
\item[Case 1)]
there exists an open neighbourhood $s \in Z^0 \sub Z$ and a unique section $\tilde{e}\colon Z^0 \ra \on{Sing}(C_{Z^0}/Z^0)$ such that $(\tilde{e})_s = e$. Then define $\phi_E(e) = (\tilde{e})_\zeta$; 
\item[Case 2)] case 1 does not hold \emph{and} (writing $v_1$, $v_2$ for the endpoints of $e$) we have that $\phi_V(v_1) = \phi_V(v_2)$. Then map $e$ to $\phi_V(v_1)$. 
\end{enumerate}
In fact, case 2 holds if and only if the label $l(e)$ is a unit at $\zeta$. 
\end{definition}

%The map $\phi \colon \Gamma_s \ra \Gamma_\zeta$ defined above is surjective on vertices and edges. In particular, this means that the labelled graph $\Gamma_\zeta$ can be obtained from the labelled graph $\Gamma_s$ by contracting exactly those edges whose labels become units in $\ca{O}_{S.\zeta}$, and replacing the labels on the other edges by their images under $\on{sp}\colon \ca{O}_{S,s} \hra \ca{O}_{S,\zeta}$. 

A more intuitive description of the specialisation morphism $\phi \colon \Gamma_s \ra \Gamma_\zeta$ on labelled graphs may be given as follows: starting with $\Gamma_s$, first replace each label by its image under $\on{sp}$. Then contract every edge whose label is a unit. This is exactly the labelled graph $\Gamma_\zeta$. Given that such a simple description is available, why did we give the long-winded definition above? Essentially this is because it is otherwise not a-priori clear that the labelled graph resulting from this simple description is (naturally) isomorphic to the labelled graph $\Gamma_\zeta$. 

%As we have this simple description, why did we 

%\begin{lemma}
%Let $C/S$ be a quasisplit nodal curve, and let $s$ be a geometric point of $S$ such that $C/S$ is aligned at $s$. Then there exists an open Zariski  neighbourhood $V_s$ of $s$ in $S$ such that for all geometric points $v$ of $V_s$, $C$ is aligned over $v$. 
%\end{lemma}

\section{Centrally controlled curves}\label{sec:controlled_curves}

\begin{definition}Let $C/S$ be a nodal curve. A point $\frak{s} \in S$ is called a \emph{controlling point} for $C/S$ if the following conditions hold:
\begin{enumerate}
\item $S$ is affine and noetherian, and $C/S$ is quasisplit;
\item for each edge $e$ of $\Gamma_\frak{s}$, the intersection of $\on{label}(e)$ with $\ca{O}_S(S)$ inside the local ring $\ca{O}_{S,\frak{s}}$ is a principal ideal;
\item
for every point $s$ in $S$, there exists a point $\eta_s \in S$ such that
\begin{enumerate}
\item both $s$ and $\frak{s}$ are in the closure of $\eta_s$;
\item the specialisation map $\Gamma_s \ra \Gamma_{\eta_s}$ is an isomorphism on the underlying graphs. 
\end{enumerate}
\end{enumerate}
\end{definition}
By (2), we can think of the labels of $\Gamma_\frak{s}$ as being principal ideals in $\ca{O}_S(S)$. We will generally reserve lowercase fraktur letters $\frak{s}$ and $\frak{t}$ for controlling points. 

\begin{definition}
Let $C/S$ be a nodal curve over a locally noetherian scheme. Let $\tau$ be in $\{$smooth, \'etale$\}$. A \emph{controlled $\tau$-cover} of $C/S$ consists of a collection $(S_i, \frak{s}_i)_{i \in I}$ of pointed schemes and a map $\bigsqcup_{i\in I} S_i \ra S$ such that 
\begin{enumerate}
\item $\bigsqcup_{i\in I} S_i \ra S$ is a cover in the $\tau$-topology;
\item for each $i$, the point $\frak{s}_i$ is a controlling point for $C \times_S S_i$. 
\end{enumerate}
\end{definition}

\begin{lemma}\label{lem:controlled_after_shrink}
Let $C/S$ be a quasisplit curve over a locally noetherian scheme, and $\frak{s} \in S$ a point. Then there exists an open neighbourhood $V$ of $\frak{s}$ in $S$ such that $\frak{s}$ is a controlling point for $C_V/V$. 
\end{lemma}
\begin{proof}
Shrinking $S$, we may assume that $S$ is affine and that every label on the graph $\Gamma_\frak{s}$ is generated by an element of $\ca{O}_{S}(S)$ (i.e. these locally-principal ideals are principal). Then delete from $S$ every irreducible component that does not contain $\frak{s}$.

Let $\Sigma$ denote the smallest collection of (reduced) closed subsets of $S$ which is closed under:
\begin{itemize}
\item
pairwise intersections;
\item taking irreducible components;
\end{itemize}
and which contains the image in $S$ of $\on{Sing}(C/S)$. 

Now set 
$$Z \defeq \bigcup\{\sigma \in \Sigma | \frak{s} \notin \sigma\}, $$
the union of all elements of $\Sigma$ which don't contain $\frak{s}$. Note this is a closed subset since $\Sigma$ is finite. Let $V$ denote the complement of $Z$ in $S$.

Now let $s \in V$ be any point. Suppose first that $\Gamma_s$ consists of a single vertex and no edges - this is equivalent to saying that $s$ is not contained in any element of $\Sigma$, or to saying that $C_s/s$ is smooth. Now the locus where $C_V/V$ is smooth is open in $V$ and is non-empty (since it contains $s$). Since every irreducible component of $V$ contains $\frak{s}$ by assumption. Hence we can find a point $\eta_s$ with $C_{\eta_s}/\eta_s$ smooth and such that both $\frak{s}$ and $s$ are specialisations of $\eta_s$. 

Suppose now that $C_s/s$ is not smooth. Let $\sigma \in \Sigma$ be the smallest element of $\Sigma$ which contains $s$ (note that $\sigma$ contains $\frak{s}$ by construction), and let $\eta_s$ be the generic point of $\sigma$. Then it is clear that both $\frak{s}$ and $s$ are contained in the closure of $\eta_s$. It remains to see that the specialisation map 
$$\on{sp}\colon \Gamma_s \ra \Gamma_{\eta_s}$$
is an isomorphism on the underlying graphs. This is equivalent to checking that no label of $\Gamma_s$ is mapped to a unit in $\ca{O}_{S, \eta_s}$. Well, any label $l$ on an edge of $\Gamma_s$ which becomes a unit at $\eta_s$ will cut out a proper closed subscheme of $\sigma$ containing $s$, but this is impossible by the definition of $\sigma$. 
\end{proof}

\begin{lemma}\label{lem:controlled_cover_exists}
Let $C/S$ be a nodal curve over a locally noetherian base. Then there exists an \'etale controlling cover for $C/S$. 
\end{lemma}
\begin{proof}
By \cref{lem:quasisplit_after_etale_cover}, we may assume $C/S$ is quasisplit, and it is enough to show that every point $\frak{s} \in S$ has an open neighbourhood for which $\frak{s}$ is a controlling point. We are done by \cref{lem:controlled_after_shrink}. 
\end{proof}

%Both maps are surjective on the underlying graphs by construction. The key point is that the map $\on{sp}_t$ is in fact an \emph{isomorphism} on the underlying graphs. This is easy to see: the map $\on{sp}_t$ simply contracts every edge of $\Gamma_t$ whose label is a unit at $\eta$. Well, any label $l$ on an edge of $\Gamma_t$ which becomes a unit at $\eta$ will cut out a proper closed subscheme of $\sigma$ containing $t$, but this is impossible by the definition of $\sigma$. 
%Now let $H$ be any maximal 2-vertex-connected subgraph of $\Gamma'_t$. Let $H'$ be the corresponding maximal 2-vertex-connected subgraph of $\Gamma'_\eta$. On the `graphs without loops', we can make a map $\on{sp}_s'\colon \Gamma_s'\ra\Gamma'_\eta$ by first including $\Gamma'_s$ into $\Gamma_s$, then applying $\on{sp}_s$, then contracting down any loops in $\Gamma_\eta$. Since the map $\on{sp}'_s$ simply adds some loops and then contracts some edges, we see that there exists a 2-vertex-connected subgraph $H''$ of $\Gamma'_s$ such that $H'$ is contained in its image. Then the K$ contains generators for each of the labels of $H$. Conversely, given an element $K = K(H)$ of $K$, either all elements of $K(H)$ are units at $t$ (in which case the corresponding component $H$ of $\Gamma_s$ is contracted by $\on{sp}_s$), or the image of $H$ in $\Gamma'_t$ is a 2-vertex-connected component whose labels are the pullback of $K(H)$ to $\ca{O}_{S,t}$. 

\newcommand{\cseq}{\mathfrak{K}}

\section{Construction of universal aligning morphisms}\label{sec:scheme_representing_relations}
%\subsection{The definition}
%\begin{definition}\label{def:universal_aligning_scheme}
%Let $C/S$ be a nodal curve over a reduced locally noetherian scheme, and let $U$ be the largest open subscheme of $S$ such that $C$ is smooth over $U$. 
%\begin{enumerate}
%\item A morphism $f\colon T \ra S$ is \emph{non-degenerate} if $T$ is reduced and $f^*U$ is dense in $T$. 
%\item A non-degenerate morphism $f\colon T \ra S$ is an \emph{aligning morphism} for $C/S$ if the pullback $f^*C$ is aligned. 
%\item A morphism $\beta \colon \tilde{S} \ra S$ is a \emph{universal aligning morphism for $C/S$} if every aligning morphism for $C/S$ factors uniquely via $\beta$. 
%\end{enumerate}
%A universal aligning morphism is unique up to unique isomorphism if it exists. If $C/S$ is aligned and $U$ is dense in $S$ then $\on{id}\colon S \ra S$ is a universal aligning morphism for $C/S$.  
%\end{definition}

Now we will prove existence of universal aligning morphisms, and later some properties. 
\subsection{The case of controlled curves}
Throughout this section we fix a nodal curve $C/S$ and a controlling point $\frak{s} \in S$. 

A \emph{circuit} in a graph is a path which starts and ends at the same point, and otherwise does not repeat any edge or vertex. 

\begin{definition}
Let $G = (V,E, \on{ends})$ be a finite graph. To any subset $E_0 \sub E$ we associate the unique subgraph of $G$ with edges $E_0$ and with no isolated vertices - we will often fail to distinguish between $E_0$ and the subgraph. We say $E_0$ is \emph{circuit-connected} if for every pair $e$, $e'$ of distinct edges in $E_0$ there is at least one circuit $\gamma \sub E_0$ such that $e \in\gamma$ and $e' \in \gamma$. 
\end{definition}
\begin{lemma}
The maximal circuit-connected subsets of $E$ form a partition of $E$. 
\end{lemma}
\begin{proof}
Let $e \in E$, then $\{e\}$ is circuit-connected. It suffices to show that if $\bb{E}$ and $\bb{E}'$ are both circuit-connected and both contain $e$, then the union $\bb{E} \cup \bb{E}'$ is also circuit-connected. If $a$ and $b$ are edges in $\bb{E} \cup \bb{E}'$, then let $\gamma_a$ denote a circuit in $\bb{E} \cup \bb{E}'$ containing both $a$ and $e$, and let $\gamma_b$ be a circuit in $\bb{E} \cup \bb{E}'$ containing both $b$ and $e$. We will construct a circuit in $\bb{E} \cup \bb{E}'$ containing both $a$ and $b$. 

We will do this by `splicing' $b$ into $\gamma_a$. Let $p_0$ and $p_1$ be the ends of $b$ (necessarily distinct unless $a = b$). Let $\gamma_0$ be the shortest sub-path path of $\gamma_b$ which starts at $p_0$, does not contain $b$, and which meets $\gamma_a$ (say at a point $q_0$). Define $\gamma_1$ and $q_1$ similarly. Let $\gamma'$ denote the sub-path of $\gamma_a$ which goes between $q_0$ and $q_1$ and which contains $a$. Then the union of $\gamma'$, $\gamma_0$, $\gamma_1$ and $b$ is a circuit containing $a$ and $b$. 
%Pick one endpoint $a_0$ of $a$, and orient $\gamma_a$ to lead away from $a$ at $a_0$. Proceeding along $\gamma_a$ in the chosen direction, let $p$ be the first point encountered such that $p \in \gamma_b$ and such that following $\gamma_b$ in some direction will lead from $p$ to $b$ \emph{without} meeting $\gamma_a$ again. Now follow $\gamma_b$ from $p$ in that direction, going through $b$ and carrying on out of the other side. Let $q$ be the first point where you meet $\gamma_a$ after leaving $b$. Then $q$ must lie strictly after $p$ (in the chosen orientation on $\gamma_a$), so following $\gamma_a$ in that chosen orientation will lead back to $a$, and we have just traversed a circuit in $\bb{E} \cup \bb{E}'$ containing $a$ and $b$. 
\end{proof}
We write $\on{Part}(G)$ for this partition. We can give a couple of alternative descriptions of this partition: 
\begin{enumerate}
\item
For each loop $l$, $\{l\}\in\on{Part}(G)$. For each 2-vertex-connected component $H$ of the graph obtained from $G$ by deleting loops, $\on{Part}(G) \ni \on{edges}(H)$ (cf. \cite{Holmes2014Neron-models-an}). 
\item 
For each $e \in E$ let $X(e)$ denote the union of $\{e\}$ with all the other edges which are contained in any common circuit with $e$. Then if two $X(e)$ have non-empty intersection they must be equal, so in this way we obtain a partition of the edges of $G$, and it is exactly the partition given above. 
\end{enumerate}

\begin{definition}[Thickness function]
Let $E$ denote the set of edges of the graph $\Gamma_{\frak{s}}$. A \emph{thickness function} is a function 
\begin{equation*}
M \colon E \ra \bb{Z}_{\ge 0}
\end{equation*}
satisfying the following condition: 

\emph{Let $\Gamma_M$ be the graph obtained from $\Gamma_\frak{s}$ by contracting every edge $e$ such that $M(e) = 0$. We then require that for each set $\bb{E} \in \on{Part}(\Gamma_M)$, we have $\gcd(M(\bb{E})) = 1$. }
\end{definition}

We now give three examples of controlled curves and the possible thickness functions. 
\begin{example} Let $S = \on{Spec}\bb{C}[[u,v]]$, and let $\frak{s}$ be the closed point. Assume $\frak{s}$ is a controlling point, and suppose the graph over $\frak{s}$ is

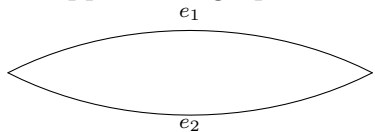
\definecolor{cqcqcq}{rgb}{0.7529411764705882,0.7529411764705882,0.7529411764705882}
\begin{tikzpicture}[line cap=round,line join=round,>=triangle 45,x=0.5991249375662954cm,y=0.4636246466940381cm]
\clip(-9.05517135775812,-2.420524881055739) rectangle (10.974039992958119,2);
\draw [shift={(0.0,6.0)}] plot[domain=4.124386376837122:5.3003915839322575,variable=\t]({1.0*7.211102550927979*cos(\t r)+-0.0*7.211102550927979*sin(\t r)},{0.0*7.211102550927979*cos(\t r)+1.0*7.211102550927979*sin(\t r)});
\draw [shift={(0.0,-6.0)}] plot[domain=0.982793723247329:2.1587989303424644,variable=\t]({1.0*7.211102550927979*cos(\t r)+-0.0*7.211102550927979*sin(\t r)},{0.0*7.211102550927979*cos(\t r)+1.0*7.211102550927979*sin(\t r)});
\begin{scriptsize}
\draw[color=black] (0,-1.5) node {$e_2$};
\draw[color=black] (0,1.69) node {$e_1$};
\end{scriptsize}
\end{tikzpicture}

\noindent with $\on{label}(e_1) = (u)$, $\on{label}(e_2) = (v)$. Then thickness functions are exactly those functions which send the two edges to non-negative coprime integers, e.g.
\begin{equation*}
M(e_1) = 2,\;\;\;\;\; M(e_2) = 3
\end{equation*}
 or 
 \begin{equation*}
M(e_1) = 0,\;\;\;\;\; M(e_2) = 1.
\end{equation*}
\end{example}

\begin{example}Let $S = \on{Spec}\bb{C}[[u,v,w]]$, and let $\frak{s}$ be the closed point. Assume $\frak{s}$ is a controlling point, and suppose the graph over $\frak{s}$ is

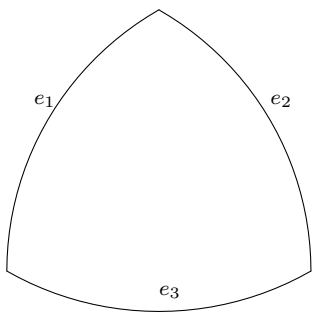
\definecolor{uuuuuu}{rgb}{0.26666666666666666,0.26666666666666666,0.26666666666666666}
\begin{tikzpicture}[line cap=round,line join=round,>=triangle 45,x=0.5cm,y=0.5cm]
\clip(-9.05517135775812,-5.6) rectangle (10.974039992958119,4);
\draw [shift={(-4.0,-4.0)}] plot[domain=0.0:1.0471975511965976,variable=\t]({1.0*8.0*cos(\t r)+-0.0*8.0*sin(\t r)},{0.0*8.0*cos(\t r)+1.0*8.0*sin(\t r)});
\draw [shift={(4.0,-4.0)}] plot[domain=2.0943951023931953:3.141592653589793,variable=\t]({1.0*8.000000000000002*cos(\t r)+-0.0*8.000000000000002*sin(\t r)},{0.0*8.000000000000002*cos(\t r)+1.0*8.000000000000002*sin(\t r)});
\draw [shift={(0,2.92820323027551)}] plot[domain=4.1887902047863905:5.235987755982988,variable=\t]({1.0*8.000000000000002*cos(\t r)+-0.0*8.000000000000002*sin(\t r)},{0.0*8.000000000000002*cos(\t r)+1.0*8.000000000000002*sin(\t r)});
\begin{scriptsize}
\draw[color=black] (3.220796889455059,0.502745483489056) node {$e_2$};
\draw[color=black] (-3,0.502745483489056) node {$e_1$};
\draw[color=black] (0.29632799953723776,-4.564066895554612) node {$e_3$};
\end{scriptsize}
\end{tikzpicture}

\noindent with $\on{label}(e_1) = (u)$, $\on{label}(e_2) = (v)$ and $\on{label}(e_3) = (w)$. Then again the thickness functions are exactly those functions which send the three edges to non-negative integers with no common factor, e.g.
 \begin{equation*}
M(e_1) = 2,\;\;\;\;\; M(e_2) = 3, \;\;\;\;\; M(e_3) = 5
\end{equation*}
or
 \begin{equation*}
M(e_1) = 0,\;\;\;\;\; M(e_2) = 2, \;\;\;\;\; M(e_3) = 3 
\end{equation*}
or
 \begin{equation*}
M(e_1) = 0,\;\;\;\;\; M(e_2) = 0, \;\;\;\;\; M(e_3) = 1. 
\end{equation*}
\end{example}

\begin{example}
 Let $S = \on{Spec}\bb{C}[[u,v,w]]$, and let $\frak{s}$ be the closed point. Assume $\frak{s}$ is a controlling point, and suppose the graph over $\frak{s}$ is

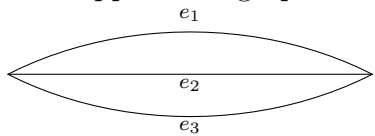
\definecolor{cqcqcq}{rgb}{0.7529411764705882,0.7529411764705882,0.7529411764705882}
\begin{tikzpicture}[line cap=round,line join=round,>=triangle 45,x=0.5991249375662954cm,y=0.4636246466940381cm]
\clip(-9.05517135775812,-2.420524881055739) rectangle (10.974039992958119,2);
\draw [shift={(0.0,6.0)}] plot[domain=4.124386376837122:5.3003915839322575,variable=\t]({1.0*7.211102550927979*cos(\t r)+-0.0*7.211102550927979*sin(\t r)},{0.0*7.211102550927979*cos(\t r)+1.0*7.211102550927979*sin(\t r)});
\draw [shift={(0.0,-6.0)}] plot[domain=0.982793723247329:2.1587989303424644,variable=\t]({1.0*7.211102550927979*cos(\t r)+-0.0*7.211102550927979*sin(\t r)},{0.0*7.211102550927979*cos(\t r)+1.0*7.211102550927979*sin(\t r)});
\draw (-4.0,0.0)-- (4.0,0.0);
\begin{scriptsize}
\draw[color=black] (0,-1.5) node {$e_3$};
\draw[color=black] (0,1.69) node {$e_1$};
\draw[color=black] (0,-0.3) node {$e_2$};
\end{scriptsize}
\end{tikzpicture}

\noindent with $\on{label}(e_1) = (u)$, $\on{label}(e_2) = (v)$ and $\on{label}(e_3) = (w)$ Then every function $M$ which sends the three edges to \emph{positive} integers with no common factor is a thickness function, but if $M$ sends one of the $e_i$ to $0$ then the other two $e_j$ become loops in the graph $\Gamma_M$ obtained by contracting $e_i$, and so must be sent to $1$. Thickness functions are exactly those sending the $e_i$ to positive coprime integers, and those sending $(e_1, e_2, e_3)$ to some permutation of $(0,1,1)$ and $(0,0,1)$. 
\end{example}

We will now set up some notation which we will use repeatedly for the remainder of this paper. 
\begin{setup}\label{def:standard_notation}\leavevmode
\begin{enumerate}
\item
$C/S$ is a nodal curve with controlling point $\frak{s}$;
\item We write $R = \ca{O}_S(S)$ (so $S = \on{Spec} R$); 
\item $U\hra S$ is the largest open in $S$ over which $C$ is smooth, and we assume $U$ is dense in $S$;
\item $M$ is a thickness function, and $\Gamma_M$ is the graph obtained from $\Gamma_\frak{s}$ by contracting every edge $e$ such that $M(e) = 0$;
\end{enumerate}
\end{setup}

\begin{definition}\label{def:rings_of_relations}
Notation as in \cref{def:standard_notation}, and let $\bb{E} \in \on{Part}(\Gamma_M)$. Write $M(e) = m_e$, and write $\ell(e) \in R$ for the label of an edge $e$. Let $n_\star\colon \bb{E} \ra \bb{Z}$ be a function such that $\sum_{e \in \bb{E}} n_em_e = 1$ - this is possible by the coprimality condition. 
Define
$$R'_\bb{E} = \frac{R[a, u_e^{\pm1}:e \in \bb{E}]}{(\ell(e) - a^{m_e}u_e : e \in \bb{E}, \:\: 1-\prod_{e\in\bb{E}} u_e^{n_e})}. $$
This depends on the $n_e$ (so the notation is not good). However, this dependence will turn out not to matter (see \cref{prop:universal_M_aligning}). To simplify the notation, we will assume that a choice of $n_e$ has been made once-and-for-all for every collection of $m_e$. 

%A-priori this depends on the choice of the $n_i$, but by \cref{prop:universal_M_aligning} we see that $R_M'$ is in fact determined up to unique isomorphism independently of the choice of $n_i$. 
We have a natural map
\beqs
\begin{split}
R'_\bb{E} & \ra \ca{O}_{U}(U)\\
a & \mapsto \prod_{e \in \bb{E}} \ell(e)^{n_e}\\
u_e & \mapsto \ell(e)a^{-m_e}. 
\end{split}
\eeqs
%The ideal $(a)$ is the aligning ideal. 
We call the element $a \in R'_\bb{E}$ the \emph{aligning element} of $R_\bb{E}'$. %\chck{really?!}
\end{definition}

\begin{definition}
Notation as in \cref{def:standard_notation}. Define 
$$R'_M = \bigotimes_{\bb{E} \in \on{Part}(\Gamma_M)} R'_{\bb{E}}.  $$
By the universal property of the tensor product we get a natural map 
$$R'_M \ra \ca{O}_{U}(U), $$
and we define $R_M$ to be the image of $R'_M$ under the above map. Set $S_M = \on{Spec}R_M$ (in other words, $S_M$ is the closure of the image of $U$ in $\on{Spec}R_M'$ under the given map). Write $a_\bb{E}$ for the image in $R_M$ of the aligning element in $R_{\bb{E}}'$, then set $a := (a_\bb{E})_{\bb{E} \in \on{Part}(\Gamma_0)}$. 
\end{definition}

\begin{remark}%Omid Amini is thinking about something like this
In the above definition, the procedure of `take the product, then take the image of the result in $\ca{O}_{U}(U)$'  is essentially a way to `splice together' all the rings of relations $R'_{\bb{E}}$ without worrying about how the relations they capture are connected to each other. In general this image $R_M$ is not very explicit (for example, it is not immediately apparent how to write down generators and relations for it, given the same for the $R_{\bb{E}}$). On the other hand, if the singularities of $C/S$ are `mild enough' (for example, in the universal case) then we find that $R_M = R_M'$ - see \cref{lemma:NCD_preverved}. 
%Making this $R_M$ more explicit in this sense may be an interesting question. 
\end{remark}

% Given a scheme $X$ and elements $a$, $b \in \ca{O}_X(X)$, we write $a \sim b$ (or $a \sim_X b$) if there exists a unit $u \in \ca{O}_X(X)^\times$ such that $au = b$. 
 
\begin{definition}
Notation as in \cref{def:standard_notation}. Let $f\colon T \ra S$ be a non-degenerate morphism. We say $f$ is \emph{$M$-aligning} if for all $\bb{E} \in \on{Part}(\Gamma_M)$ there exists $a_\bb{E} \in \ca{O}_T(T)$ such that for all $e \in \bb{E}$ we have $f^*\on{label}(e)  =  (a_\bb{E})^{M(e)}$. 
\end{definition}

\begin{lemma}\label{lem:loc_M_ali}
Notation as in \cref{def:standard_notation}. Let $f\colon T \ra S$ be an aligning morphism. Then there exists an \'etale cover $\bigsqcup_{i \in I} T_i \ra T$, and for each $i$ a thickness function $M_i$, such that each $T_i \ra S$ is $M_i$-aligning. 
\end{lemma}
\begin{proof}
By \cref{lem:quasisplit_after_etale_cover} we may assume $T$ is quasisplit. Let $\frak{t} \in T$ be any point. Shrinking $T$, we may assume by \cref{lem:controlled_after_shrink} that $\frak{t}$ is a controlling point for $f^*C/T$. We are done if we can find a thickness function $M$ such that, after possibly shrinking $T$ further, we have that $C_T/T$ is $M$-aligning. 

%We know by \cref{lem:controlled_cover_exists} that $T$ has a controlled \'etale cover, so we reduce immediately to the case where $T$ itself is controlled, say by a point $\frak{t} \in T$. Moreover, we can shrink $T$ further (as long as we do not delete $\frak{t}$). \check{this comes from looking at proof of existence of controlling covers, but is it clear from statement?}

Let $L_{\on{units}}$ denote the set of labels of edges of $\Gamma_\frak{s}$ which are units at $f(\frak{t}) \in S$ - this is exactly the same as the set of labels of edges of $\Gamma_\frak{s}$ which are contracted by the natural map to $\Gamma_{f(\frak{t})}$ (cf. \cref{def:contraction_map}). Because $\frak{t}$ is controlling for $f^*C/T$, we see that $L_{\on{units}}$ is also exactly the set of labels of edges of $\Gamma_\frak{s}$ which pull back to units on $T$ (recalling that we can think of the labels of $\Gamma_\frak{s}$ as principal ideals in $R$). The thickness function $M$ we will construct will take the value $0$ exactly on edges in $L_{\on{units}}$. 

Because $f$ is aligning, we know that for each set $\bb{E} \in \on{Part}(\Gamma_\frak{t})$ there exists a principal ideal $a_\bb{E} \triangleleft \ca{O}_{T, \frak{t}}$ such that every edge in $\bb{E}$ is labelled by some power of $a_\bb{E}$ - write $\on{label}(e) = a_\bb{E}^{m_e}$. Moreover, replacing $a_\bb{E}$ by some positive power, we may assume for each $\bb{E}$ that 
\begin{equation*}
\gcd\{m_e:e\in\bb{E}\} = 1. 
\end{equation*}
Define $M(e) = m_e$ if $\on{label}(e) \notin L_{\on{units}}$, and $M(e) = 0$ otherwise. 

Now shrinking $T$ we may assume that each $a_\bb{E}$ stays a principal ideal in $\ca{O}_T(T)$, and moreover that for every $\bb{E}$ and every $e \in \bb{E}$, the relation $\on{label}(e) = a_\bb{E}^{M(e)}$ holds globally on $T$. 
\end{proof}

We define a universal $M$-aligning morphism to $S$ in an analogous manner to the definition of the universal aligning morphism, except that it depends on the controlling point $\frak{s}\in S$ and the thickness function $M$. 

\begin{proposition}\label{prop:universal_M_aligning}
Notation as in \cref{def:standard_notation}. The natural map $S_M \ra S$ is a universal $M$-aligning morphism. % with respect to $a$. % (i.e. is terminal in the category of $M$-aligning morphisms with respect to fixed sets of elements\check{awkward!}). 
\end{proposition}
\begin{proof}
It is clear that $S_M \ra S$ is $M$-aligning. Let $f\colon T \ra S$ be any $M$-aligning morphism. The uniqueness of a factorisation of $f$ via $S_M$ is clear as $f$ is non-degenerate and $S_M \ra S$ is affine and hence separated. Moreover, a factorisation via $S_M'$ yields a factorisation via $S_M$ since $U$ is dense in $T$. Fix $\bb{E} \in \on{Part}(\Gamma_{M})$; by the universal property of the fibre product it is enough to construct a factorisation of $f$ via $\on{Spec}R'_\bb{E}$. %By uniqueness of factorisations it is enough to construct a factorisation after taking an \'etale cover of $T$, so we may assume $T$ has a controlling point $\frak{t}$. 

By definition of $f$ being $M$-aligning, there is an element $a \in \ca{O}_{T}(T)$ such that for all $e \in \bb{E}$, we have $f^*\on{label}(e) = (a)^{M(e)}$. Choose $n_e : e\in \bb{E}$ such that $\sum_{e \in \bb{E}} M(e) n_e = 1$. Then we define an $R$-algebra map
\beqs
\begin{split}
R'_\bb{E} & \ra \ca{O}_{T}(T)\\
a &\mapsto \prod_{e \in \bb{E}} (f^*\on{label}(e))^{n_e}\\
u_e & \mapsto f^*\on{label}(e)/a^{M(e)}
\end{split}
\eeqs
(note $ \prod_{e \in \bb{E}} (f^*\on{label}(e))^{n_e} = (a)$). 
This yields a map $T \ra \on{Spec} R'_\bb{E}$ as required. 
\end{proof}

% With old, `local' definition:
%Write $K = \{x_1, \cdots, x_n\}$. By definition of being $M$-aligning, we have a cover $T' \ra T$ (write $f'\colon T' \ra S$ for the composite map) and an element $t \in \ca{O}_{T'}(T')$ such that $t^{M(I_j)} \sim_{T'} f'^*x_{j}$. Then we define an $R$-algebra map
%\beq
%\begin{split}
%R'_M & \ra \ca{O}_{T'}(T')\\
%a &\mapsto t\\
%u_i & \mapsto f'^*x_i/a^{M(I_i)}. 
%\end{split}
%\eeq
%This yields a map $T' \ra \on{Spec} R'_M$, which in fact factors uniquely via $S_M$ as $f$ is non-degenerate. It remains to descend this map $T' \ra S_M$ to a map $T \ra S_M$. This works because $S_M$ is separated and $f^*U_\cseq$ is by definition scheme-theoretically-dense in $T$, and so morphisms are uniquely determined by what they do over $U$. 

\begin{definition}\label{def:SMN}
Notation as in \cref{def:standard_notation}. Let $M$, $N$ be two thickness functions. Define
\begin{equation*}
\delta_M = \bigcup \{\bb{E} \in \on{Part}(\Gamma_M) : \exists e \in \bb{E} \text{ with }M(e) \neq N(e)\}
\end{equation*}
and
\begin{equation*}
\delta_N = \bigcup \{\bb{E} \in \on{Part}(\Gamma_N) : \exists e \in \bb{E} \text{ with }M(e) \neq N(e)\}, 
\end{equation*}
then set $\delta_{M,N} = \delta_M \cup \delta_N$. Then define
\beqs
S_{M,N} = \on{Spec} R_M[\on{label}(e)^{-1}:e\in \delta_{M,N}]
\eeqs
(here we slightly abuse notation by adjoining inverses of principal ideals), so $S_{M,N}$ is an open subscheme of $S_M$. 
\end{definition}
It is clear that $S_{M,N} \ra S$ is $M$-aligning, since it factors via $S_M \ra S$. Part of the point of the definition is that we also have:
\begin{lemma}
In the above notation:
\begin{enumerate}
\item
The morphism $S_{M,N} \ra S$ is $N$-aligning; 
\item The induced map $S_{M,N} \ra S_N$ factors via $S_{N,M} \ra S_N$. 
\end{enumerate}
\end{lemma}
\begin{proof}
\begin{enumerate}
\item
Let $\bb{E} \in \on{Part}(\Gamma_N)$. We want to show that there exists an element $a \in R$ such that for all $e \in \bb{E}$, we have $\on{label}(e) = (a)^{N(e)}$. Well, if $M|_\bb{E} \neq N|_\bb{E}$ then $\bb{E}\subseteq \delta_{M,N}$ so we may take $a=1$. On the other hand, if $M|_\bb{E} = N|_\bb{E}$ then there exists $\bb{E}_M \in \on{Part}(\Gamma_M)$ such that $\bb{E} \subseteq \bb{E}_M$, and we may take the same aligning element $a$ as works for $\bb{E}_M$. 
\item By the universal property of localisation, it is enough to show that all the functions on $S_N$ which become units on $S_{N,M}$ also become units on $S_{M,N}$, but this is clear since $\delta_{M,N} = \delta_{N,M}$. 
\end{enumerate}
\end{proof}
From this lemma and symmetry we obtain a canonical isomorphism $S_{M,N} \ra S_{N,M}$ for all pairs of thickness functions $M$, $N$. 
\begin{definition}
Notation as in \cref{def:standard_notation}. Define $\beta\colon \tilde{S} \ra S$ to be the result of glueing together all the $S_N$ as $N$  runs over all thickness functions, along the open subschemes $S_{N,N'}$. 
\end{definition}

\begin{proposition}\label{prop:universal_aligning_separated} The map $\beta\colon \tilde{S} \ra S$ is separated. 
\end{proposition}
\begin{proof}
The map $\beta\colon \tilde{S} \ra S$ is quasi-separated because $\tilde{S}$ is locally noetherian, so it is enough to check the valuative criterion for separatedness. Let $V$ denote the spectrum of a valuation ring, with generic point $\eta$ and closed point $v$. Let $f, g \colon V \ra \tilde{S}$ be morphisms which agree on $\eta$ and such that the composites with the canonical map $\tilde{S} \ra S$ agree - in particular, $f^*\on{label}(e) = g^*\on{label}(e)$ for every edge $e$ of $\Gamma_\frak{s}$. We will show that $f = g$. 

First, the $S_N$ form an open cover of $S$ as $N$ runs over thickness functions, so there exist thickness functions $M$ and $N$ such that $f$ factors via $S_M \ra S$ and $g$ factors via $S_N \ra S$. If $M=N$ then the result is clear since $S_M \ra S$ is affine and hence separated. Thus we may as well assume that $M \neq N$. 

We will show that both $f$ and $g$ factor via $S_{M,N}$, which is affine over $S$ and so $f=g$. To do this, we need to show that for every edge $e \in \delta_{M,N}$ we have that $f^*\on{label}(e)$ is a unit on $V$. Well, fix some $e_0 \in \delta_{M,N}$. Then (perhaps switching $M$ and $N$) we may assume that there exists $\bb{E} \in \on{Part}(\Gamma_M)$ such that $e \in \bb{E}$ and such that $N|_\bb{E} \neq M|_\bb{E}$ (and note that $M$ does not vanish on any edge in $\bb{E}$). Since $f$ is $M$-aligning, we know there exists $a \in \ca{O}_V(V)$ such that for all $e\in\bb{E}$, we have $f^*\on{label}(e) = (a)^{M(e)}$. 

Observe that $f^*\on{label}(e_0) \neq 0$, otherwise we cannot have $f(\eta) \in S_{M,N}$, which contradicts the fact that $f$ and $g$ agree on $\eta$. Since $M(e_0) \neq 0$, this tells us that $a \neq 0$. 

We now divide into two cases:
\begin{enumerate}
\item[Case 1.] There exists $e' \in \bb{E}$ such that $N(e') = 0$. Then $g^*\on{label}(e')$ is a unit, so $a^{M(e')}$ is a unit, so $a$ is a unit, so all the $f^*\on{label}(e)$ for $e \in \bb{E}$ are units as required. 
\item[Case 2.] $N$ does not vanish on any $e \in \bb{E}$. Then there exists $\bb{E}_N \in \on{Part}(\Gamma_N)$ such that $\bb{E} \subseteq \bb{E}_N$, so there exists $b \in \ca{O}_V(V)$ such that for all $e \in \bb{E}_N$ we have $g^*\on{label}(e) = (b)^{N(e)}$, so certainly the same holds for $\bb{E}$. Now since $M|_\bb{E} \neq N|_\bb{E}$ and $M$ does not vanish on $\bb{E}$, there exist integers $c_e :e \in \bb{E}$ such that
\begin{equation*}
d := \sum_{e \in \bb{E}} c_e M(e) \neq 0 \text{ and } \sum_{e \in \bb{E}} c_e N(e) = 0. 
\end{equation*}
Moreover, a similar argument to that above tells us that $b \neq 0$. It is enough to show that $a$ is a unit on $V$. 

Pick a generator $\ell(e)$ of the principal ideal $\on{label}(e)$ for each $e \in \bb{E}$. Then we find that, up to multiplication by units on $V$, we have
\begin{equation*}
\prod_{e \in \bb{E}}f^*\ell(e)^{c_e} = \prod_{e \in \bb{E}} a^{c_eM(e)} = a^d
\end{equation*}
and 
\begin{equation*}
\prod_{e \in \bb{E}}g^*\ell(e)^{c_e} = \prod_{e \in \bb{E}} b^{c_eM(e)} = 1. 
\end{equation*}
The left hand sides are equal, so $a^d$ is a unit on $V$, so $a$ is a unit on $V$ and we are done. 
\end{enumerate}
\end{proof}

\begin{lemma}\label{lem:uam_controlled}
Notation as in \cref{def:standard_notation}. Then $\beta\colon\tilde{S} \ra S$ is a universal aligning morphism for $C/S$. The map $\beta$ is locally of finite presentation. 
\end{lemma}
\begin{proof}
Let $f\colon T \ra S$ be aligning. The uniqueness of a factorisation of $f$ via $\beta$ holds because $f$ is non-degenerate and $\beta$ is separated by \cref{prop:universal_aligning_separated}. The existence comes from combining \cref{lem:loc_M_ali} and \cref{prop:universal_M_aligning} with descent. 

Local finite presentation holds because the $S_M$ are clearly of finite presentation. 
\end{proof}
\subsection{Universal aligning morphisms: the general case}

\begin{theorem}\label{thm:existence_of_iniversal_aligning_for_curve_over_scheme}%\chck{maybe combine this with the lemma just above? Else seems lots of redundancy in the proof? }
Let $C/S$ a generically-smooth nodal curve over a reduced locally noetherian algebraic stack. Then a universal aligning morphism for $C/S$ exists, and is a separated algebraic space locally of finite presentation over $S$. 
\end{theorem}
\begin{proof}
The stack $S$ admits a smooth cover by a scheme (which is necessarily reduced and locally noetherian). A universal aligning morphism will descend along a smooth cover (as will the property of being an algebraic space) since it is defined by a universal property. As such, it is enough to consider the case where $S$ is a scheme. Similarly, by \cref{lem:controlled_cover_exists} we can assume that $C/S$ has a controlling point $\frak{s}$. The existence then follows from \cref{lem:uam_controlled} and effectivity of descent for algebraic spaces. Separatedness and local finite presentation follow by \'etale descent of those properties, and the same lemma. 
\end{proof}

\section{Regularity and normal crossings}\label{sec:regularity_for_correlation}

\begin{definition}[Normal crossings singularities]
Notation as in \cref{def:standard_notation}. We say $C/S$ has \emph{normal crossings singularities} if the sequence $(\on{label}(e):e \in \Gamma_\frak{s})$ has normal crossings - in other words, if for every set $J \subseteq \on{edges}(\Gamma_\frak{s})$, we have that the closed subscheme
\begin{equation*}
V(\on{label}(e):e \in J)\subseteq S
\end{equation*}
is regular and has codimension $\#J$ in $S$ at every point in that subscheme. 
\end{definition}
\begin{definition}[\'Etale normal crossings singularities]
Let $C/S$ be a nodal curve over a locally noetherian scheme. We say $C/S$ has \emph{\'etale normal crossings singularities} if for every geometric point $\bar{s}$ of $S$, and for every subset $J \sub \on{edges}(\Gamma_{\bar{s}})$, the closed subscheme 
\begin{equation*}
V(\on{label}(e):e \in J) \sub \on{Spec}\ca{O}_{S,\bar{s}}
\end{equation*}
is regular and has codimension $\#J$. 
\end{definition}
Note that having \'etale normal crossings singularities is smooth-local on the target, and so makes sense when $S$ is an algebraic stack. 
\begin{lemma}\label{lem:controled_etale_cover_NCS}
Let $C/S$ be a nodal curve over an excellent scheme with \'etale normal crossings singularities. Then there exists an \'etale controlled cover $\bigsqcup_{i \in I} S_i \ra S$ (with controlling points $\frak{s}_i \in S_i$) such that the pullback of $C$ to each $S_i$ has normal crossings singularities. 
\end{lemma}
\begin{proof}
This is clear by \cref{lem:controlled_after_shrink}, the finiteness of the sets of edges, and the openness of the regular locus in an excellent scheme. 
%
%By \crer{?} we may reduce to the case where $C/S$ is quasi split. Let $\frak{s} \in S$ be a point; we will show that $\frak{s}$ has a Zariski open neighbourhood on which $C$ squires normal crossings singularities. By \cref{?} \check{really by the proof - adjust} we know $\frak{s}$ has a Zariski neighbourhood where the 
\end{proof}

Suppose $C/S$ has \'etale normal crossings singularities. In particular, by taking $J$ to be empty, we deduce that $S$ must be regular. It is then easy to check that $C$ is also regular. On the other hand, $C$ can be regular without having \'etale normal crossings singularities - for example if $S$ is a trait and $C$ has multiple non-smooth points. Finally, note that the universal stable curve $\Mbar{g,n+1} \ra \Mbar{g,n}$ has \'etale normal crossings singularities. 
%Note that having normal crossings is smooth-local on the target. 

\begin{lemma}\label{lemma:basic_regularity} Let $R$ be a ring, and $x_1, \cdots, x_d \in R$ a collection of elements such that for all $J \sub \{1, \cdots, d\}$, the quotient 
\begin{equation*}
R/(x_j:j \in J)
\end{equation*}
is regular. Let $m_1, \cdots, m_d$ be non-negative integers with $\gcd(m_1, \cdots, m_d) = 1$, and let $n_1, \cdots, n_d$ be integers such that $\sum_{i=1}^dm_in_i = 1$. Define 
\begin{equation*}
R' = \frac{R[a, u_1^{\pm1}, \cdots, u_d^{\pm1}]}{(x_i - a^{m_i}u_i : (1 \le i \le d),  1-\prod_i u_i^{n_i})}. 
\end{equation*}
Then $R'$ and $R'/a$ are regular. 
\end{lemma}
\begin{proof}
It is easy to check that $R'/a$ is regular; it is even smooth over $R/(x_i : m_i >0)$, which is regular by assumption. We need to show $R'$ itself is regular; this will take more care, since it is not in general smooth over its image - it resembles an affine patch of a blowup. 

Let $p \in \on{Spec}R'$ be any point, and write $q$ for the image of $p$ in $\on{Spec}R$. Localising $R$ at $q$, we may assume that $R$ is local, with closed point $q$. Re-ordering the $x_i$, we may assume that $x_1, \cdots, x_e \in q$ and $x_{e+1}, \cdots, x_d\notin q$ for some $0 \le e \le d$. Writing $D = \on{dim} R$, our normal crossing assumptions imply that there exist $g_1, \cdots, g_{D-e} \in R$ such that 
$$q = (x_1, \cdots, x_e, g_1, \cdots, g_{D-e}). $$
Now if $e = 0$ then the result is clear since $R'$ is smooth over $R$. Hence we may assume $e \ge 1$. It then follows that $m_i = 0$ for every $e < i \le d$, otherwise $R'/qR'$ is empty, contradicting the existence of $p$. Again reordering, we may assume that $1\le m_1 \le m_2 \le \cdots \le m_e$. We find that 
$$R' /qR' = \frac{R/q[a, u_1^{\pm1}, \cdots, u_d^{\pm1}]}{(a^{m_1}, x_{e+1}-u_{e+1}, \cdots, x_d - u_d, 1-\prod_{1 \le i \le d} u_i^{n_i})}. $$
We then see that 
$$\frac{R'}{((a) + q)R'} = \frac{R/q[a, u_1^{\pm1}, \cdots, u_d^{\pm1}]}{(a, x_{e+1}-u_{e+1}, \cdots, x_d - u_d, 1-\prod_{1 \le i \le d} u_i^{n_i})}$$
is regular and of dimension $e-1$. From this we deduce that there exist elements $f_1, \cdots, f_{e-1} \in R'$ such that 
\beqs
\begin{split}
p & = (a, f_1, \cdots, f_{e-1}, x_1, \cdots, x_e, g_1, \cdots, g_{D-e})\\
& = (a, f_1, \cdots, f_{e-1}, g_1, \cdots, g_{D-e}), 
\end{split}
\eeqs
so $p$ can be generated by $D$ elements. Now it is clear that every irreducible component of $R'$ has dimension at least $D$ (count generators and relations), and hence it follows that $R'$ is regular at $p$ and has pure dimension $D$. 
\end{proof}

\begin{lemma}\label{lem:connected_fibre}
In the notation of \cref{lemma:basic_regularity}, assume also that $R$ is local, with maximal ideal $\frak{m}$. Assume that at least one of the $x_i$ lies in $\frak{m}$. The following are equivalent:
\begin{enumerate}
\item
$R'/\frak{m}$ is connected;
\item $R'/\frak{m}$ is non-zero;
\item for all $1 \le i \le d$, we have that $\left( x_i \in R^\times \implies m_i = 0\right)$. 
\end{enumerate}
\end{lemma}
\begin{proof}
The implication $1 \implies 2$ is from the definition of a connected ring. Suppose 3 fails, so (say) $x_1 \in R^\times$ and $m_1 \neq 0$, so $a$ becomes a unit in $R'/\frak{m}$. We also know some $x_i \notin R^\times$, say $x_d \notin R^\times$. Then in $R'/\frak{m}$ we find that $0 = x_d = a^{m_d}u_d$ is a unit, so $R'/\frak{m} = 0$. This shows $2 \implies 3$. 

For $3 \implies 1$ we must work a little harder. A little more notation: let $k = R/\frak{m}$, and assume that $x_1, \cdots, x_e \notin R^\times$ and that $x_{e+1}, \cdots, x_d \in R^\times$ for some $e > 0$. Then
\begin{equation*}
\begin{split}
R'/\frak{m} &= \frac{k[a, u_1^{\pm 1}, \cdots, u_d^{\pm 1}]}{\left(u_1a^{m_1}, \cdots, u_ea^{m_e}, x_{e+1} - u_{e+1}, \cdots, x_d - u_d, 1-\prod_{i=1}^du_i^{n_i}\right)}\\
& = \frac{k[a, u_1^{\pm 1}, \cdots, u_e^{\pm 1}]}{\left(a^{m_1}, \cdots, a^{m_e}, 1-\prod_{i=1}^eu_i^{n_i}\prod_{i=e+1}^dx_i^{n_i}\right)}. \\
\end{split}
\end{equation*}
Quotienting by $(a)$ will not affect whether the ring is connected, so it is enough to show that 
\begin{equation*}
\frac{k[u_1^{\pm 1}, \cdots, u_e^{\pm 1}]}{\left(1-\prod_{i=1}^eu_i^{n_i}\prod_{i=e+1}^dx_i^{n_i}\right)}
\end{equation*}
is connected. We will in fact show that this ring is a domain. Perhaps extending the field $k$, we can absorb the $x_i$. Moreover since $\sum_{i=1}^dm_in_i = \sum_{i=1}^em_in_i = 1$ we know that $n_i$ have no common factor. Perhaps swapping $u_i$ and $u_i^{-1}$ for some $i$, we may assume all $n_i \ge 0$. Moreover, the localisation of a domain is a domain. So it is enough to show the ring
\begin{equation*}
\frac{k[u_1, \cdots, u_e]}{\left(1 - \prod_{i=1}^e u_i^{n_i} \right)}
\end{equation*}
is a domain, i.e. we must show $P := 1 - \prod_{i=1}^e u_i^{n_i} $ is irreducible. 

Without loss of generality assume $n_1 \neq 0$. Think of $P$ as a polynomial in $x_1$, and for any $i \neq 1$ write $\on{NP}_{x_i}(P)$ for the Newton polygon of $P$ with respect to the valuation coming from $x_i$. The single edge of the Newton polygon $\on{NP}_{x_i}(P)$ has slope $n_i/n_1$, so we see that if $h$ is a factor of $P$ then $\on{deg}_{x_1}(h)\frac{n_i}{n_1} \in \bb{Z}$, i.e. 
\begin{equation*}
\frac{n_1}{\gcd(n_1, n_i)} \mid \on{deg}_{x_1}(h). 
\end{equation*}
Since the $n_i$ have no common factor, this implies that $\on{deg}_{x_1}(P) = n_1 \mid \on{deg}_{x_1}(h)$, and we are done. 
\end{proof}

\begin{lemma}\label{lem:dense_image}
In the notation of \cref{lemma:basic_regularity}. Let $U \sub \on{Spec}R$ be dense open such that every $x_i$ is a unit on $U$. Define $\phi\colon U \ra \on{Spec}R'$ by
\beqs
\begin{split}
R' & \ra \ca{O}_{U}(U)\\
a & \mapsto \prod_{i} x_i^{n_i}\\
u_i & \mapsto x_ia^{-m_i}. 
\end{split}
\eeqs
Then the image of $\phi$ is dense in $\on{Spec}R'$. 
\end{lemma}
\begin{proof}
Without loss of generality, we may assume $U$ is given by $U = \on{Spec}R_0$ where $R_0 = R[1/x_i:1 \le i \le d]$. Set $S = \on{Spec}R$, and $S' = \on{Spec}R'$. First, we want to show that the natural map $U \ra S' \times_S U$ is an isomorphism, in other words that $R' \otimes_R R_0 \ra R_0$ is an isomorphism. Since not all $m_i =0 $ we find that $a$ becomes a unit in $R' \otimes_R R_0$, and the result then follows by elementary manipulations. 

Now let $p \in S' \setminus U$ be a point, with image $q\in S$. So $q \notin U$, so some $x_i \in q$. Localising $R$ at $q$, the hypotheses are preserved, but now we also have that $R$ is local and $R'/q$ is non-zero (since $p$ exists). By \cref{lem:connected_fibre} this implies that the closed fibre $S'_q$ is connected. Write $\overline{\phi(U)}$ for the closure of the image of $U$ in $S'$. I claim that the fibre $\overline{\phi(U)}_q$ is non-empty. Suppose for now that the claim is true. Then $p$ lies in the same connected component of $S'$ as $\phi(U)$. But $S'$ is regular by \cref{lemma:basic_regularity}, so every connected component of $S'$ is irreducible, so $p \in \overline{\phi(U)}$ and the lemma is proven. 

It remains to verify the claim that $\overline{\phi(U)}_q\neq \emptyset$. For this, let $f\colon T \ra S$ be a map from the spectrum of a discrete valuation ring to $S$ sending the closed point to $q$ and such that for all $1 \le i \le d$ we have 
\begin{equation*}
\on{ord}_Tf^*x_i = m_i. 
\end{equation*}
This is possible because the non-unit $x_i$ form a regular sequence in $R$, and because $$x_i \in R^\times \implies m_i = 0,$$ otherwise the point $p$ could not exist (by \cref{lem:connected_fibre}). 

We then find that $\prod_{i=1}^d (f^*x_i)^{n_i}$ is a uniformiser on $T$, and that for all $i$ the element $f^*x_i/a^{m_i}$ is a unit on $T$. We can therefore lift the map $f$ to a map $f'\colon T \ra S'$ by 
\beq
\begin{split}
R' & \ra \ca{O}_{T}(T)\\
a & \mapsto \prod_{i} (f^*x_i)^{n_i}\\
u_i & \mapsto f^*x_ia^{-m_i}. 
\end{split}
\eeq
The image of the closed point of $T$ under $f'$ lies over $q$. The image of $T$ under $f'$ is integral, and the image of the generic point lies over $U$, so the closure of $\phi(U)$ has non-empty fibre over $q$ as required. 
\end{proof}

\begin{lemma}\label{lemma:NCD_preverved}
Notation as in \cref{def:standard_notation}, and assume $C/S$ has normal crossings singularities. For each $\bb{E} \in \on{Part}(\Gamma_M)$, let $a_\bb{E}$ denote the image of the aligning element of $R'_\bb{E}$ in $R'_M$. Then 
\begin{enumerate}
\item
The sequence $(a_\bb{E} : \bb{E} \in \on{Part}(\Gamma_M))$ form a normal-crossings divisor in $R'_M$ (in particular they are distinct);
\item We have $R_M = R'_M$. 
\end{enumerate}
%Let $S = \on{Spec} R$ be an affine scheme, and let $\cseq = \{K_1, \cdots, K_r\}$ a correlation collection with normal crossings. Let $M = (M_1, \cdots, M_r) \in \bb{M}(\cseq)$. For any $0 \le t \le r$, define 
%$$R_t := \bigotimes_{1\le i \le t} R'_{M_i}$$
%(so $R_0 = R)$. Write $\cseq_t' = \{K_{t+1}', \cdots, K_r'\}$ for the image of $\{K_{t+1}, \cdots, K_r\}$ in $\on{Spec} R_t$, and set $K_{t,0} = \{a_1, \cdots, a_t\}$, where $a_i$ is an aligning element for $R_M'$ (cf. \cref{def:rings_of_relations}). Then the multiset $\cseq_t := \{K_{t,0}, K'_{t+1}, \cdots, K'_r\}$ has normal crossings. %Moreover, we have $R_M' = R_M$.
%The following hold:
%\begin{enumerate}
%\item the pull-back of $\cseq_t := (K_{t+1}, \cdots, K_{r})$ to $\on{Spec}R_t$ has normal crossings. 
%\item   
%\item set $K_0 = \{a_1, \cdots, a_t\}$. Then $(K_0, K_{t+1}, \cdots, K_r)$ 
%at each step in the induction, we have that $\{a_1, \cdots, a_t\}\cup \bigcup_{t <i \le r} K_i$ form an NCD, where $a_i$ is the special uniformiser in $R_M'$ (cf. \cref{def:rings_of_relations}). 
%\end{enumerate}
%Note that 3 implies 1. 
\end{lemma}
\begin{proof}
Choose an ordering on $\on{Part}(\Gamma_M)$, say $\on{Part}(\Gamma_M) = \{\bb{E}_1, \cdots, \bb{E}_r\}$. Set $R_0 = R$, and $R_i = \bigotimes_{j=1}^iR'_{\bb{E}_j}$. We will prove the claim by induction on $i$. The case $i=0$ is exactly the assumption that $C/S$ has normal crossings singularities (in particular $U$ is dense in $R$). Suppose we know the result for some $R_i$, then claim 1 (respectively claim 2) for $R_{i+1}$ is exactly what we get by applying \cref{lemma:basic_regularity} (respectively \cref{lem:dense_image}) to the ring $R_i$ and the sequence $a_{\bb{E}_1}, \cdots, a_{\bb{E}_i}$, with suitably chosen $m_\star$ and $n_\star$. 
\end{proof}

\begin{theorem}\label{thm:Stilde_regular}
Notation as in \cref{def:standard_notation}, and assume $C/S$ has normal crossings singularities. Let $\beta\colon \tilde{S} \ra S$ be the universal aligning morphism. Then the set $\{\beta^*\on{label}(e):e \in \on{edges}(\Gamma_\frak{s})\}$ has normal crossings in $\tilde{S}$, i.e. for every $J \sub \on{edges}(\Gamma_\frak{s})$, the subscheme
\begin{equation}
V(\beta^*\on{label}(e):e \in J) \sub \tilde{S}
\end{equation}
is regular. In particular, $\tilde{S}$ is regular. 
\end{theorem}
\begin{proof}
The claim is local on $\tilde{S}$, so we can fix a thickness function $M$ and check the claim on $S_M$. The result is then immediate from \cref{lemma:NCD_preverved}. 
%The analogous statement for $S_M'$ is exactly \cref{lemma:NCD_preverved}. Note that the map $S_M \ra S$ becomes an isomorphism after base-change over $S$ to $U_\cseq$, which implies that the closure of the image of $U_\cseq$ in $\on{Spec}R_M'$ is an irreducible component of $\on{Spec}R_M'$\check{check}. Since $R_M'$ is regular, this closure must in fact be a connected component, and must itself be regular. This closure is $S_M$ by definition. 
%Given $J \sub \on{edges}(\Gamma_\frak{s})$, the subscheme $V(\beta^*\on{label}(e):e \in J) \sub S_m$ is a union of connected components\check{check} of a subscheme of $\on{Spec}R_M'$ cut out by a collection of elements of $\{\beta^*\on{label}(e):e \in J\}$; in particular, it is again regular since these define a normal crossings divisor (by \cref{lemma:NCD_preverved}). 
\end{proof}

\section{Resolving singularities over the universal aligning scheme}\label{sec:resolving_sings_of_pullback}

\begin{definition}
Let $S$ be a scheme, $U \sub S$ a dense open, and $C/U$ a smooth proper curve. A \emph{model} for $C/U$ is a proper flat morphism $\bar{C} \ra S$ together with an isomorphism $\bar{C} \times_S U \stackrel{\sim}{\longrightarrow} C$. This isomorphism will often be suppressed in the notation. 
\end{definition}

\begin{lemma}\label{lemma:resolving_singularities}
Notation as in \cref{def:standard_notation}, and assume $C/S$ has normal crossings singularities.  Then the pull-back of $C_U$ to $S_M = \on{Spec}R_M$ has a nodal \textbf{regular} aligned model. 
\end{lemma}
Note that $C$ is regular since $C/S$ has normal crossings singularities. 
\begin{proof}
Let $f\colon S_M \ra S$ be the structure map. Note that $S_M$ is regular by \cref{thm:Stilde_regular}. 
Write $C_0 = C \times_S S_M$, this is an aligned nodal curve. We will resolve the singularities of $C_0$ by blowing up, taking care to preserve alignment as we do so. 

For $i  = 0, 1, \cdots$, we define $Z_i \sub C_i$ to be the reduced closed subscheme where $C_i$ is not regular, and then define $C_{i+1}$ to be the blowup of $C_i$ at $Z_i$. It is enough to show:
\begin{enumerate}[label={1.\arabic*}]
\item some $C_N$ is regular;% (in particular the sequence stabilises);
\item each $C_i$ is aligned. 
\end{enumerate}

%[ref={\ref{def:correlation}.\arabic*}]

%Write $\underline{\cseq} = \cup_{K \in \cseq} K$ for the set of all elements of $\cseq$. 
Write $P$ for the (finite) set of generic points of the union of the $V(f^*x)$ as $x$ runs labels of $\Gamma_\frak{s}$. Note that each $\frak{p} \in P$ is a codimension 1 point on the regular scheme $S_M$, and moreover (again by \cref{thm:Stilde_regular}) that the quotient $R/\frak{p}$ is also regular. Write $\on{ord}_\frak{p}$ for the corresponding ($\bb{Z}\cup \{\infty\}$)-valued discrete valuation on the local ring at $\frak{p}$. Given $i \ge 0$, define a non-negative integer
$$\delta_i = \sum_{\frak{p} \in P} \sum_{\stackrel{q \in C_i}{\text{ lying over }\frak{p}}}(\on{ord}_\frak{p} \ell(q) - 1); $$
here $\ell(q)$ denotes the label of the edge of the dual graph $\Gamma_\frak{p}$ corresponding to $q$; it lives in the \'etale local ring at $\frak{p}$. It is enough to show:
\begin{enumerate}[label={2.\arabic*}]
\item if $\delta_i>0$ then $\delta_{i}>\delta_{i+1}$; \label{d_drops}
\item if $\delta_i = 0$ then $C_i$ is regular; \label{regular_if_d_zero}
\item each $C_i$ is aligned. \label{item:preserves_alignment}
\end{enumerate}

We first show \cref{d_drops}. Let $i$ be such that $\delta_i > 0$. Let $\frak{p} \in P$, and let $q \in C_i$ be a non-smooth point lying over $\frak{p}$, with $\on{ord}_\frak{p}\ell(q) = t$. Let $p$ be a generator for $\frak{p}$. The completed local ring on $C_i$ at $q$ is given by
$$\widehat{\ca{O}}_{C_i,q}\cong \frac{\widehat{\ca{O}}_{S_M,\frak{p}}[[u,v]]}{(uv-p^t)}. $$
Assume $t \ge 2$ (this holds for some $\frak{p}$ and $q$, otherwise $\delta_i = 0$). The blowup of $\widehat{\ca{O}}_{C_i,q}$ at $q = (u,v, p)$ has three affine patches: 
\begin{enumerate}[label={2.1.\arabic*}]
\item
`$u_i = 1$', given by:
$$\frac{\widehat{\ca{O}}_{S_M,\frak{p}}[[u,v]][v_1, p_1]}{(v_1-u^{t-2}p_1^t, v-uv_1, p-up_1)}$$
which is regular (a calculation, or cf. \cite[\S 3]{Jong1996Smoothness-semi}); 
\item `$v_1 = 1$', which is also regular by symmetry;
\item `$p_1=1$', given by 
$$\frac{\widehat{\ca{O}}_{S_M,\frak{p}}[[u,v]][u_1, v_1]}{(u_1v_1-p^{t-2}, u-pu_1, v-pv_1)}. $$
This patch is empty and hence regular if $t=2$. If $t>2$ then the patch is regular except at $q' := (u_1, v_1, p)$. In the latter case we see that $\on{ord}_\frak{p}q' = t-2$; it has dropped by 2. 
\end{enumerate}
As such, we see that at most one non-regular point $q'$ of $C_{i+1}$ maps to $q$, and if $q'$ exists we have $\on{ord}_\frak{p}q' = \on{ord}_\frak{p}q - 2$. This shows that $\delta_{i+1} < \delta_i$. 

Next we show \cref{regular_if_d_zero}. Let $i$ be such that $\delta_i = 0$. Let $c \in C_i$ lie over $s \in S_M$. We will show $C_i$ is regular at $c$. If $C_i \ra S_M$ is smooth at $c$ we are done, so assume this is not the case. Then the completed local ring of $C_i$ at $c$ is given by
$$\frac{\widehat{\ca{O}}_{S_M,s}[[u,v]]}{(uv-\ell_c)}$$
where $\ell_c \in \ca{O}_{S,s}$ is an element which (by definition) generates the label of the graph $\Gamma_s$ at the edge $e_c$ corresponding to the point $c$. 

Let $\frak{p} \in P$ be such that $\ell_c \in \frak{p}$ (such $\frak{p}$ exists by construction). The specialisation map 
$$\Gamma_s \ra \Gamma_\frak{p}$$
does not contract $e_c$; rather it sends it to an edge labelled by the ideal generated by $\on{sp}(\ell_c)$, where 
$$\on{sp}\colon \ca{O}_{S_M,s} \ra \ca{O}_{S_M,\frak{p}}$$
is the specialisation map. By our assumption that $\delta_i = 0$, it follows that $(\on{sp}\ell_c) = \frak{p}$, so the closed subscheme $V(\ell_c) \sub \on{Spec}\ca{O}_{S_M,s}$ is regular. 

Write $\on{dim} \ca{O}_{S_M,s} = d$. By the above regularity statement, we can find elements $g_1, \cdots, g_{d-1}$ such that 
$$\frak{m}_s = (\ell_c, g_1, \cdots, g_{d-1}). $$
Hence the ideal corresponding to $c$ can be generated by
$$(u,v,\ell_c, g_1, \cdots, g_{d-1}) = (u,v, g_1, \cdots, g_{d-1}), $$
in other words it can be generated by $d+1$ elements. Since $\on{dim}_c C_M = d+1$, this proves that $C_M$ is regular at $c$. 

Finally, we show \cref{item:preserves_alignment}. We proceed by induction on $i$. For $i=0$ the result is clear. Let $i \ge 1$, and assume the result for $i-1$. Let $s \in S_M$ be any point, and for each $i$ let $\Gamma^i_s$ be the graph of $C_{i}$ over $s$. Then the labelled graph $\Gamma^i_s$ can be constructed from the labelled graph $\Gamma_s^{i-1}$ by the following recipe:

\begin{enumerate}
\item for each edge $e$ such that $\on{label}(e) = a^2$ for some irreducible element $a \in \ca{O}_{S_M,s}^{\on{et}}$, replace $e$ by two edges both with label $a$;
\item
 for each edge $e$ such that $\on{label}(e) = a^n$ for some irreducible element $a \in \ca{O}_{S_M,s}^{\on{et}}$ and integer $n >2$, replace $e$ by three edges with labels $a$, $a^{n-2}$ and $a$ in that order;
 \item delete any edge labelled by a unit. 
\end{enumerate}
 In pictures:

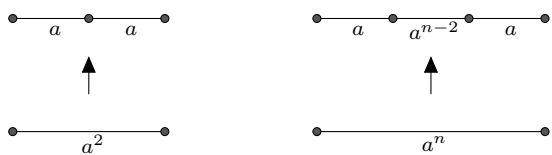
\definecolor{qqqqff}{rgb}{0.3333333333333333,0.3333333333333333,0.3333333333333333}
\begin{tikzpicture}[line cap=round,line join=round,>=triangle 45,x=1.0cm,y=1.0cm]
\clip(-6,-0.5) rectangle (6.,2);
\draw (-4.0,0.0)-- (-2.0,0.0);
\draw (-4.0,1.5)-- (-3.0,1.5);
\draw (-3.0,1.5)-- (-2.0,1.5);
\draw (0.0,-0.0)-- (3.0,-0.0);
\draw (0.0,1.5)-- (1.0,1.5);
\draw (1.0,1.5)-- (2.0,1.5);
\draw (2.0,1.5)-- (3.0,1.5);
\draw [<-] (-3.0,1.0) -- (-3.0,0.5);
\draw [<-] (1.5,1.0) -- (1.5,0.5);
\begin{scriptsize}
\draw [fill=qqqqff] (-4.0,0.0) circle (1.5pt);
\draw[color=qqqqff] (-3.8606621971293755,0.2982758459192761) node {};
\draw [fill=qqqqff] (-2.0,0.0) circle (1.5pt);
\draw[color=qqqqff] (-1.8561816847495791,0.2982758459192761) node {};
\draw[color=black] (-2.9419419622886354,-0.1610842715010952) node {$a^2$};
\draw [fill=qqqqff] (-4.0,1.5) circle (1.5pt);
\draw[color=qqqqff] (-3.8606621971293755,1.8016362302041273) node {};
\draw [fill=qqqqff] (-3.0,1.5) circle (1.5pt);
\draw[color=qqqqff] (-2.8584219409394773,1.8016362302041273) node {};
\draw [fill=qqqqff] (-2.0,1.5) circle (1.5pt);
\draw[color=qqqqff] (-1.8561816847495791,1.8016362302041273) node {};
\draw [fill=qqqqff] (0.0,-0.0) circle (1.5pt);
\draw[color=qqqqff] (0.14829882763021732,0.2982758459192761) node {};
\draw [fill=qqqqff] (3.0,-0.0) circle (1.5pt);
\draw[color=qqqqff] (3.1550195961999123,0.2982758459192761) node {};
\draw [fill=qqqqff] (0.0,1.5) circle (1.5pt);
\draw[color=qqqqff] (0.14829882763021732,1.8016362302041273) node {};
\draw [fill=qqqqff] (1.0,1.5) circle (1.5pt);
\draw[color=qqqqff] (1.1505390838201155,1.8016362302041273) node {};
\draw [fill=qqqqff] (2.0,1.5) circle (1.5pt);
\draw[color=qqqqff] (2.152779340010014,1.8016362302041273) node {};
\draw [fill=qqqqff] (3.0,1.5) circle (1.5pt);
\draw[color=qqqqff] (3.1550195961999123,1.8016362302041273) node {$$};
\draw[color=black] (-3.4430620903835845,1.3422761127837561) node {$a$};
\draw[color=black] (-2.4408218341936863,1.3422761127837561) node {$a$};
\draw[color=black] (1.5472591852286168,-0.1610842715010952) node {$a^{n}$};
\draw[color=black] (0.5450189290387187,1.3422761127837561) node {$a$};
\draw[color=black] (1.5472591852286168,1.3422761127837561) node {$a^{n-2}$};
\draw[color=black] (2.5494994414185155,1.3422761127837561) node {$a$};
\draw[color=black] (-2.879301946276767,0.9246760060379642) node {};
\draw[color=black] (1.6307792065777753,0.9246760060379642) node {};
\end{scriptsize}
\end{tikzpicture} \\
It is clear that applying this procedure to an aligned graph will yield an aligned graph. 
\end{proof}

%The `contraction' morphism $C_{i} \ra C_{i-1}$ induces a map on labelled graphs, which on underlying graphs simply contracts some edges. In particular, the image of a 2-vertex-connected component lies in a 2-vertex-connected component\chck{false}. Let $G$ be a 2-vertex connected component of $\Gamma_s^i$, and let $H$ be a 2-vertex-connected component of $\Gamma^{i-1}_s$ containing the image of $G$. 
%
% By our induction hypothesis, there exists an $a \in \ca{O}_{S_M,s}$ such that every label on $H$ is a power of $a$; say each edge $e$ is labelled by $a^{n(e)}$. By our normal-crossing assumptions, we may assume $V(a)$ is a regular subscheme of $\on{Spec}\ca{O}_{S_M,s}$. Then $G$ is a sub-labelled-graph of the labelled graph obtained from $H$ by, for each edge $e$ with $n(e) \ge 2$, replacing $e$ by three edges with labels $a$, $a^{n(e)-2}$ and $a$ respectively (and deleting any edge labelled by a unit). In pictures:
% 
% In particular, we see that $G$ is a subgraph of a graph all of whose labels are powers of $a$, and hence $G$ has the same property. 
% 
\section{Existence of N\'eron models over universal aligning schemes}\label{sec:NM_over_univ_aligning}

\begin{definition}\label{def:universal_blowup}
Let $C/S$ be a nodal curve over an algebraic stack, smooth over a dense open substack $U \sub S$. Write $J$ for the jacobian of $C_U/U$; this is an abelian scheme over $U$. A \emph{N\'eron-model-admitting morphism for $C/S$} is a morphism $f\colon T \ra S$ of algebraic stacks such that:
\begin{enumerate}[label={\ref{def:universal_blowup}.\arabic*}]
\item $T$ is regular;
%\item $f$ is representable by schemes; Not any more!
\item $U \times_S T$ is dense in $T$;
\item $f^*J$ admits a N\'eron model over $T$. 
\end{enumerate}
We define the category of N\'eron-model-admitting morphisms as the full sub-2-category of the 2-category of stacks over $S$ whose objects are N\'eron-model-admitting morphisms. 
%Morphisms of N\'eron-model-admitting morphisms are morphisms over $\Mbar{g,n}$. \chck{rephrase!}

A \emph{universal N\'eron-model-admitting morphism for $C/S$} is a terminal object in the 2-category\footnote{Since this is a 2-category, care must be taken with terminal objects. However, since our categories are fibred in groupoids there is no ambiguity in the definition of terminal object. All our proofs will be by descent from the case of schemes, so this issue will also not arise in the proofs. } of N\'eron-model-admitting morphisms; it is unique up to isomorphism unique up to unique 2-isomorphism if it exists. We write $\beta\colon \tilde{S} \ra S$ for a universal N\'eron-model-admitting morphism, and $\tilde{N}/\tilde{S}$ for the N\'eron model of $J$.% \chck{check 2-cat stuff}
%of regular algebraic stacks such that $\beta$ is terminal in the category of N\'eron-model-admitting morphisms. We write $\MN{g,n}/\Mtil{g,n}$ for the N\'eron model of $J_{g,n}$. 
\end{definition}

%Note that excellence is local on the source NO!

\begin{theorem}\label{thm:NM_over_univ_aligning_over_scheme}
Let $S$ be an algebraic stack locally of finite type over an excellent scheme. Let $C/S$ be a nodal curve such that $C/S$ has \'etale normal crossing singularities (in particular, $C$ is smooth over some dense open substack $U \sub S$). Let $\beta\colon \tilde{S} \ra S$ denote the universal regular aligning morphism (note $\beta$ is an isomorphism over $U$). Then $\tilde{S}$ is a universal N\'eron-model-admitting morphism for $C/S$, and moreover the N\'eron model $\tilde{N}$ is of finite type over $\tilde{S}$, and its fibrewise-connected-component-of-identity is semi-abelian. 
%regular and the Jacobian $J$ of $C_U/U$ admits a finite-type N\'eron model over $\tilde{S}$. 
\end{theorem}
Note that the universal aligning morphism exists as an algebraic space over $S$ by \cref{thm:existence_of_iniversal_aligning_for_curve_over_scheme}. 

\begin{proof}
Let $f\colon S' \ra S$ be a smooth cover by a scheme, then $f^*C/S'$ has \'etale normal crossings singularities, and $S'$ is excellent because it is locally of finite type over an excellent scheme. So by \cref{lem:controled_etale_cover_NCS} we find that $S$ has a smooth cover by schemes $S_i : i \in I$ such that each $C_{S_i} \ra S_i$ has a controlling point $\frak{s}_i$ and normal crossings singularities. Formation of the universal aligning morphism $\tilde{S} \ra S$ commutes with smooth base-change, so by \cref{thm:Stilde_regular} we find that $\tilde{S}$ is regular, and by \cref{lemma:resolving_singularities} we find that $\tilde{S}$ has a smooth cover by schemes $S_{i,M}$ such that on each $S_{i,M}$ the curve $C_{S_{i,M}}$ has a nodal regular aligned model. 

By \cite[theorem 1.2]{Holmes2014Neron-models-an} we then know that the jacobian of $C_U/U$ has a finite type N\'eron model over each $S_{i,M}$ with semi-abelian fibrewise-connected-component-of-identity. N\'eron models descend along smooth covers by \cite[lemma 6.1]{Holmes2014Neron-models-an}, so a N\'eron model exists over $\tilde{S}$ (and the properties mentioned also descend). This shows that $\beta\colon \tilde{S} \ra S$ is a N\'eron model admitting morphism. 

%The fact that $C$ is regular and has normal-crossing singularities descends along smooth surjective base-change over $S$. The universal aligning morphism descends along fppf morphisms, and regularity of it descends along smooth surjective base-change over $S$. Moreover, N\'eron models descend along smooth surjective morphisms by \cite[lemma 6.1]{Holmes2014Neron-models-an}. Since every algebraic stack has by definition a smooth cover by a scheme, it suffices to prove the result in the case where $S$ is an affine scheme. 
%Next, we will show that $\tilde{S} \ra S$ is a N\'eron-model-admitting morphism. By \cref{thm:regularity_of_universal_aligning} we know that $\tilde{S}$ is regular. Combining the same descent arguments as above with \cref{prop:correlation_sequence_exists_after_etale_cover}, we may assume $C/S$ has a correlation collection $\cseq$. Then $\tilde{S}$ has a Zariski cover by opens $S_M$ as $M$ runs over $\bb{M}(\cseq)$. Again using that N\'eron models descend along smooth (and hence Zariski) covers, it is enough to construct a N\'eron model for the pullback of $J$ to $S_M$ for each $M$. Fix an $M \in \bb{M}(\cseq)$. By \cref{lemma:resolving_singularities} we know that the pullback of $C_U$ to $S_M$ admits a nodal regular aligned model. Then by \cite[theorem 1.2]{Holmes2014Neron-models-an} we know that the pullback of $J$ to $S_M$ admits a finite-type N\'eron model as desired. 

Finally, we need to show that any other N\'eron-model-admitting morphism factors via $\tilde{S}$. Let $f\colon T \ra S$ be a N\'eron-model-admitting morphism. Since $T$ is reduced and $\beta$ is separated, a factorisation of $f$ via $\beta$ will descend along an fppf-cover of $T$; as such, (and using that a scheme smooth over a regular base is regular) we may assume $T$ is a scheme. Then by \cite{Holmes2014Neron-models-an} and the fact that $f^*C_U$ has a N\'eron model over $T$, we know that $f^*C/T$ is aligned (cf. \cref{remark:reg_irreg_aligned}). Since $T$ is regular, this implies that $f^*C/T$ is aligned, and hence $f$ factors uniquely via the universal aligning morphism as required. 
\end{proof}

%Now apply it to the universal case!
%\newcommand{\bd}[1]{{\mathbf{#1}}}
%\newcommand{\M}[1]{\ca{M}_{#1}}
%\newcommand{\Mbar}[1]{\overline{\ca{M}}_{#1}}
%\newcommand{\Mtil}[1]{\widetilde{\ca{M}}_{#1}}
%\newcommand{\MN}[1]{{\ca{N}}_{#1}}

From now on, we work relative to a fixed base scheme $\Lambda$ which we assume to be regular and excellent - the basic examples to keep in mind are $\on{Spec} \bb{Z}$ and the spectrum of a field. Let $g$, $n$ be non-negative integers such that $2g - 2 + n >0$. We write $\M{g,n}$ for the modui stack of smooth proper connected $n$-pointed curves of genus $g$ over $\Lambda$, and $\Mbar{g,n}$ for its Deligne-Mumford-Knudsen compactification. By \cite[theorem 2.7]{Knudsen1983The-projectivitii} we know that $\Mbar{g,n}$ is a smooth proper Deligne-Mumford stack over $\Lambda$, and the boundary $\Mbar{g,n} \setminus \M{g,n}$ is a divisor with normal crossings relative to $\Lambda$ in the sense of \cite[definition 1.8]{Deligne1969The-irreducibil}. We write $J_{g,n}$ for the jacobian of the universal curve $\Mbar{g,n+1}\times_{\Mbar{g,n}} \M{g,n}$ over $\M{g,n}$; this is an abelian scheme over $\M{g,n}$. 

%Note that a universal N\'eron-model-admitting morphism is unique if it exists. 

\begin{corollary}\label{thm:main}
A universal N\'eron-model-admitting morphism for $\Mbar{g,n}$ exists and is an algebraic space over $\Mbar{g,n}$. The N\'eron model over it is of finite type, and its fibrewise-connected-component-of-identity is semi-abelian. 
\end{corollary}
%The proof essentially consists of collecting together various results from elsewhere in this paper and from \cite{Holmes2014Neron-models-an}, and applying descent arguments. 
\begin{proof}
Since $\Mbar{g,n}$ and $\Mbar{g,n+1}$ are smooth over $\Lambda$, they are in particular regular. By \cref{thm:NM_over_univ_aligning_over_scheme}, it is enough to show that $\Mbar{g,n+1}/\Mbar{g,n}$ has \'etale normal crossing singularities. Let $s$ be a geometric point of $\Mbar{g,n}$, and let $S$ be an \'etale scheme-neighbourhood of $s$, and write $\Gamma_s$ for the labelled dual graph of the fibre over $s$ of the nodal curve $C := \Mbar{g,n+1} \times_{\Mbar{g,n}}S$. By the discussion after the proof of \cite[proposition 1.5]{Deligne1969The-irreducibil}, all the labels of $\Gamma_s$ are distinct, and they form a relative normal crossing divisor over $\Lambda$. Since $\Lambda$ is regular, these two properties together imply that $C/S$ has \'etale normal crossing singularities at $s$. %Applying this argument to each geometric point of $\Mbar{g,n}$, and using that `having \'etale normal-crossing-singularities' is smooth-local on the target, we are done. Being locally of finite type descends from the same property in \cref{thm:NM_over_univ_aligning_over_scheme}. 
\end{proof}
%We collect various nice properties of the universal blowup:

\begin{proposition}\label{prop:properties_of_universal_aligning}
These objects also satisfy the following properties:
\begin{enumerate}%[label={\ref{prop:properties_of_universal_aligning}.\arabic*}]
\item \label{item:properties}the morphism $\beta\colon \Mtil{g,n} \ra \Mbar{g,n}$ is separated and locally of finite presentation;
%\item \label{item:NM_pullback}in the notation of \cref{item:NMuniv}, if $f$ factors as $f = \beta \circ g$ for some $g \colon T \ra \Mtil{g,n}$ then $g^*\MN{g,n}$ is the N\'eron model of $f^*J_{g,n}$; \chck{Not sure this is true?}
\item \label{item:VCP_DVR}the morphism $\beta\colon \Mtil{g,n} \ra \Mbar{g,n}$ satisfies the valuative criterion for properness for morphisms from traits to $\Mbar{g,h}$ which map the generic point of the trait to $\M{g,n}$. %, and moreover the pullback of $\MN{g,n}$ along the factorisation map gives the N\'eron model of the pullback of $J_{g,n}$. 
\end{enumerate}
\end{proposition}
\begin{proof}\leavevmode
\begin{enumerate}
\item
Follows from the same properties for the universal aligning scheme (\cref{thm:existence_of_iniversal_aligning_for_curve_over_scheme}) and the fact that these properties descend along fppf morphisms.
\item Immediate from the definition of the universal N\'eron-model-admitting scheme because N\'eron models always exists over traits (alternatively, because nodal curves over traits are always aligned). 

\end{enumerate}
\end{proof}

\subsection{Base-change and smoothness}
If $\Lambda$ is a perfect field then $\Mtil{g,n}$ is automatically smooth over $\Lambda$ since it is regular. The purpose of this section is to show that the same holds when $\Lambda$ is not assumed to be a field. %Along the way we will prove a lemma about pulling back $\Mtil{g,n}$ which may also be of interest. 

Because we want to talk about how $\Mtil{g,n}$ changes as we change the base ring $\Lambda$, we will in this section write $\Mtil{g,n}^\Lambda \ra \Mbar{g,n}^\Lambda$ for the universal N\'eron model admitting morphism over $\Lambda$. 

\begin{theorem}
Let $\Lambda$ be a regular scheme. Then the structure morphism $\pi\colon \Mtil{g,n}^\Lambda \ra \Lambda$ is smooth. 
\end{theorem}
\begin{proof}
Clearly $\pi$ is locally of finite presentation. We will show that $\pi$ is flat and has regular geometric fibres. The regularity of the fibre over a geometric point $\on{Spec}k \ra \Lambda $ is immediate from \cref{lem:tilde_pullback} and the regularity of $\Mtil{g,n}^k$. For flatness we will apply \cite[6.1.5]{Grothendieck1965EGA.IV.2}. To ease notation, write $X = \Mtil{g,n}^\Lambda$. We may replace $\Mbar{g,n}^{\Lambda}$ and $\Mtil{g,n}^\Lambda$ by compatible \'etale covers by schemes, and so we will for the remainder of the proof abuse notation by treating them as schemes. Let $x$ be a point of $X = \Mtil{g,n}^\Lambda$ lying over a point $y$ of $\Lambda$. Because $X = \Mtil{g,n}^\Lambda$ and $\Lambda$ are both regular, it is enough by \cite[6.1.5]{Grothendieck1965EGA.IV.2} to show that 
\begin{equation*}
\on{dim} \ca{O}_{X,x} = \on{dim}\ca{O}_{\Lambda, y} + \on{dim}\ca{O}_{X_y, x}. 
\end{equation*}
Now $\ca{M}_{g,n}^\Lambda$ is by construction dense in the regular scheme $X = \Mtil{g,n}^\Lambda$, so $\on{dim} \ca{O}_{X,x} = \on{dim}\ca{O}_{\Lambda, y} + 3g-3 + n$. On the other hand, $\ca{M}_{g,n}^y$ is dense in the regular scheme $X_y = \Mtil{g,n}^y$ (\cref{lem:tilde_pullback} again), so $\on{dim}\ca{O}_{X_y, x} = 3g-3+n$, and we are done. %Here we have applied results on schemes to Deligne-Mumford stacks, but we can work \'etale locally so this does not cause problems.  
\end{proof}

\section{A worked example}\label{sec:example}
In this section we give an example of a non-aligned nodal curve. We compute its universal aligning scheme, and describe the N\'eron model of its jacobian over the universal aligning scheme. 

Construct a stable 2-pointed curve over $\bb{C}$ by glueing two copies of $\bb{P}^1_\bb{C}$ at $(0:1)$ and $(1:0)$, and marking the point $(1:1)$ on each copy. Then define $C/S$ to be the universal deformation as a 2-pointed stable curve. Choose coordinates such that $S = \on{Spec}\bb{C}[[x,y]]$, and $C$ is smooth over the open subset $U = D(xy)$. Call the sections $p$ and $q$. 

Now the graph over the closed point of $S$ is a 2-gon, with one edge labelled by $(x)$ and the other by $(y)$. The graph over the generic point of $(x=0)$ is a 1-gon with edge labelled by $(y)$, and similarly the graph over the generic point of $(y=0)$ is a 1-gon with edge labelled by $(x)$. All other fibres are smooth. In particular, $C/S$ is aligned except at the closed point, which is a controlling point.

\subsection{The universal aligning morphism}

We now describe the universal aligning morphism. We will not follow through the construction given in \cref{sec:scheme_representing_relations}, but will instead give a more geometric picture via a sequence of blowups of $S$. 

 Set $S_0 = S$, and let $D_0$ and $E_0$ be the divisors given by $x=0$ and $y=0$ respectively. Let $Z_0$ denote the locus where $D_0 \cup E_0$ is singular (i.e. the closed point of $S$). Let $U_0 = S_0 \setminus Z_0$. 

Now set $S_1$ to be the blowup of $S_0$ at $Z_0$, and let $D_1$ and $E_1$ be the pullbacks of $D_0$ and $E_0$ to $S_1$. Let $Z_1$ be the locus where $D_1 \cup E_1$ is singular, and let $U_1 = S_1 \setminus Z_1$. Let $\phi_1\colon U_0 \ra U_1$ be the unique $S$-morphism (an open immersion). We proceed like this, inductively defining an infinite chain of blowups $S_{i+1} \ra S_i$ and open immersions $\phi_i\colon U_i \ra U_{i+1}$. Then the universal aligning morphism $\tilde{S} \ra S$ is the colimit of the open immersions $\phi_i$. The morphism $\tilde{S} \ra S$ is separated, locally of finite type, is an isomorphism outside the closed point of $S$, but is not quasi-compact (the closed fibre is an infinite union of copies of $\bb{G}_m$). Note that $\tilde{S}$ is integral and is smooth over $\bb{C}$.

\section{Bounded-thickness substacks of $\Mtil{g,n}$}
In some applications it can be useful to consider only a part of the stack $\Mtil{g,n}$ - in particular, in applications where quasi-compactness is important. In this section we define certain open substacks of $\Mtil{g,n}$ with good universal properties. 

We write $\Gamma'$ for the graph obtained from $\Gamma$ by deleting all loops. 

\begin{definition}
Let $S$ be a scheme, $C/S$ a nodal curve, and $\bar{s}\in S$ a geometric point. Write $G_1, \cdots, G_n$ for the maximal 2-vertex-connected subgraphs of the graph $\Gamma'_s$. Let $e \ge 0$ be an integer. We say $C/S$ is \emph{$e$-strongly aligned} at $s$ if there exists a sequence $a_1, \cdots, a_n$ of non-zero elements of $\ca{O}^{et}_{S,s}$ such that 
\begin{enumerate}
\item the $a_i$ form a `weak normal-crossings-divisor' (i.e. for every subset $J \sub \{1, \cdots, n\}$, the subscheme 
$$V(a_j : j \in J) \sub \on{Spec}\ca{O}_{S,s}$$ is regular - we impose no condition on the codimension);
\item for each $i$ and each edge $e$ of $G_i$, there exists $0 \le r \le e$ such that 
$$\on{label}(e) = (a_i)^r. $$ 
\end{enumerate}
%We call a sequence such as $a_1, \cdots, a_n$ a \emph{$d$-generating sequence for the singularities of $C/S$ at $s$}. 
We say $C/S$ is \emph{$e$-strongly aligned} if it is $e$-strongly aligned at $s$ for all geometric points $s$ of $S$. 
\end{definition}

\begin{definition}\label{def:Mle}
Let $g$, $n$ be non-negative integers such that $2g - 2 + n >0$. Let $e \ge 0$ be an integer. Choose an \'etale cover $\bigsqcup_{i \in I}S_i \ra \Mbar{g,n}$ by a scheme such that on each $S_i$ the curve has a controlling point $\frak{s}_i$. For each $i \in I$, write $\bb{M}^{\le e}_i$ for the set of thickness functions on $C_{\frak{s}_i}$ which take values in $\{0, \cdots, e-1, e\} \sub \bb{Z}_{\ge 0}$. Now define $\tilde{S_i}^{\le e}$ to be the open subscheme of $\tilde{S_i}$ covered by the $S_M$ as $M$ runs over $\bb{M}_i^{\le e}$. Define $\Mtil{g,n}^{\le e}$ to be the image of $\bigsqcup_{i \in I}\tilde{S}_i^{\le e}$ under the natural \'etale map $\bigsqcup_{i \in I}\tilde{S}_i \ra \Mtil{g,n}$. 
\end{definition}

\begin{remark}
\begin{enumerate}
\item
A-priori the definition of $\Mtil{g,n}^{\le e}$ depends on the choice of cover $\bigsqcup_{i \in I}\tilde{S}_i \ra \Mbar{g,n}$, but we will see in \cref{lem:lee} that this is not the case; 
\item Note that each $\tilde{S}_i^{\le e} \ra S_i$ is quasi-compact, since $\bb{M}_i^{\le e}$ is finite. By \'etale descent of quasi-compactness, the same holds for $\Mtil{g,n}^{\le e} \ra \Mbar{g,n}$. 
\end{enumerate}
\end{remark}

%A combinatorial definition of $\Mtil{g,n}^{\le e}$ can be given. Let $S$ be a scheme. Given a correlation set $K$ on $S$, write $\bb{M}^{\le e}(K)$ for the subset of maps $M \colon K \ra \bb{Z}_{\ge 0}$ in $\bb{M}(K)$ whose image is contained in $\{0, \cdots, e-1, e\} \sub \bb{Z}_{\ge 0}$. Given a correlation collection $\cseq = \{K_1, \cdots, K_n\}$ on $S$, let 
%\begin{equation*}
%\bb{M}^{\le e}(\cseq) = \prod_i \bb{M}^{\le e}(K_i).
%\end{equation*}
%Now define $\tilde{S}^{\le e}$ to be the open subscheme of $\tilde{S}$ covered by the $S_M$ as $M$ runs over $\bb{M}^{\le e}(\cseq)$. Note that $\tilde{S}^{\le e} \ra S$ is flat if and only if it an open immersion if and only if $e=0$ or $\cseq$ is aligned.  
%
%Now the stack $\Mtil{g,n}$ is constructed, after a sequence of base-change and descent operations, from a scheme $\tilde{S}$ coming from a correlation collection. If the same procedure is applied to $\tilde{S}^{\le e}$ in place of $\tilde{S}$, then the resulting algebraic space over $\Mbar{g,n}$ will be exactly $\Mtil{g,n}^{\le e}$. 

\begin{lemma}\label{lem:lee}
Fix integers $g$, $n$, $e$ as above. Then the category of $e$-strongly aligned stable $n$-pointed curves of genus $g$ has a terminal object. This terminal object had a natural map to $\Mtil{g,n}$ (since $e$-strongly aligned implies aligned), and this map is an open immersion, whose image is exactly $\Mtil{g,n}^{\le e}$. 
\end{lemma}
\begin{proof}
In this proof we retain the notation from \cref{def:Mle} (in particular the cover $\bigsqcup_{i \in I}\tilde{S}_i \ra \Mbar{g,n}$). 

The property of being $e$-strongly aligned is local in the \'etale topology. Observe that the pullback of the canonical stable curve over $S_i$ to $\tilde{S}_i^{\le e}$ is $e$-strongly aligned, by definition of $\tilde{S}_i^{\le e}$ and the fact that the universal curve over $\Mbar{g,n}$ (and hence over $S_i$) has \'etale normal-crossings singularities. Combining these observations, we see that the pullback of the universal stable curve over $\Mbar{g,n}$ to $\Mtil{g,n}^{\le e}$ is itself $e$-strongly aligned. 

Let $f\colon T \ra \Mbar{g,n}$ be a non-degenerate morphism, with resulting stable curve $C/T$. Suppose that $C$ is $e$-strongly aligned. In particular, $C/T$ is aligned, and so $f$ factors (uniquely) via $\Mtil{g,n}$. If we can show that $f$ in fact factors via $\Mtil{g,n}^{\le e}$, then we are done. 

Since being $e$-strongly-aligned is \'etale local, we may assume that $f$ factors via the \'etale cover $\bigsqcup_{i \in I}\tilde{S}_i \ra \Mbar{g,n}$. Then it is clear from the definitions that $f$ factors via $\bigsqcup_{i \in I}\tilde{S}_i^{\le e}$ and hence via $\Mtil{g,n}^{\le e}$, as required. 
\end{proof}

Finally, we observe that the universal N\'eron model has a particularly nice universal property for 1-strongly-aligned curves:
\begin{lemma}
Let $C/S$ be stable, 1-strongly aligned and smooth over a dense open. Let $f\colon S \ra \Mtil{g,n}$ be the tautological map. Write $N_{g,n}$ for the N\'eron model over $\Mtil{g,n}$. Then $f^*N_{g,n}$ is the N\'eron model of the jacobian of $C$. 
\end{lemma}
\begin{proof}
A stable curve which is 1-strongly-aligned is necessarily regular. The Picard space of the universal curve pulls back along $f$ to the Picard space of $C/S$, and the same holds for the closure of unit section (since the latter is flat). Hence the pullback  $f^*N_{g,n}$ is the quotient of the Picard space of a regular curve by the closure of the unit section, and is hence the N\'eron model. 
\end{proof}

\section{Relation to the stack $\ca{P}_{d,g}$ of Caporaso}\label{sec:caporaso}%\chck{mention Melo's contribution}
 Let $g \ge 3$, and let $d$ be an integer such that $\on{gcd}(d-g+1,2g-2)=1$. In \cite{Caporaso2008Neron-models-an}, Caporaso constructs a smooth morphism of stacks $p\colon \ca{P}_{d,g} \ra \Mbar{g}$ and an isomorphism from $\ca{P}_{d,g} \times_{\Mbar{g}} \M{g}$ to the degree-$d$ Picard scheme of the universal curve over $\Mbar{g}$. This morphism $p$ is relatively representable by schemes, and satisfies the following property: 
 
 \emph{given a trait $B$ with generic point $\eta$ and a regular stable curve $X \ra B$, write $f\colon B \ra \Mbar{g}$ for the moduli map. Then $f^*\ca{P}_{d,g}$ is the N\'eron model of $\on{Pic}^d_{X_\eta/\eta}$. }
 
 More concisely, we might say that $\ca{P}_{d,g}$ gives a partial compactification of the degree-$d$ universal jacobian $J_{g}$ which has a good universal property for test curves $B$ in $\Mbar{g}$ which meet the boundary with `low multiplicity' (this is equivalent to the given stable curve $X/B$ being regular). In contrast, in this paper we construct a partial compactification $N_{g}$ of the universal jacobian $J_{g}$, with a natural group structure, and which has a good universal property for \emph{all} test curves to (even aligned morphisms to) $\Mbar{g}$, but at the price of `blowing up the boundary' of $\Mbar{g}$. In the remainder of this section, we will make the comparison more precise. 
 
Note that the condition $\on{gcd}(d-g+1,2g-2)=1$ precludes the possibility that $d=0$ (unless $g=2$), and so to compare Caporaso's construction to $\Mtil{g}$ we must consider N\'eron models of degree-$d$ parts of the jacobian for $d \neq 0$. These are not group schemes, but the N\'eron mapping property still makes sense. This presents no new difficulties:
\begin{theorem}
Let $S$ be a regular separated stack, and $C/S$ an aligned nodal curve, smooth over a dense open $U \hra S$. Assume that $C$ is regular. Fix $d \in \bb{Z}$, and let $J_d/U$ denote the degree-$d$ jacobian of $C$ over $U$. Then $J_d$ admits a N\'eron model over $S$. 
\end{theorem}
\begin{proof}
The closure of the unit section in $\on{Pic}_{C/S}$ coincides with the closure of the unit section in $\on{Pic}^{[0]}_{C/S}$, since the latter is open and closed in $\on{Pic}_{C/S}$. In particular, the closure of the unit section in $\on{Pic}_{C/S}$ is flat over $S$, and so the quotient $N$ of $\on{Pic}_{C/S}$ by the closure of the unit section exists as an algebraic space over $S$. It is easily verified that $N$ is the N\'eron model of $\on{Pic}_{C_U/U}$  --- the proof of the main theorem of \cite{Holmes2014Neron-models-an} carries over almost verbatim. Then the closure of $J_d$ inside $N$ is exactly the N\'eron model of $J_d$ that we seek. 
\end{proof}

%
%
%
%
%It turns out that such `low multiplicity' test curves in $\Mbar{g,n}$ can be described as test curves which lift to a suitable open part $\Mtil{g,n}^{\le 1}$ of $\Mtil{g,n}$. In the remainder of this section, we will define $\Mtil{g,n}^{\le 1}$ and show that the pullback of Caporaso's $\ca{P}_{d,g}$ to this open substack $\Mtil{g,n}^{\le 1}$ coincides with the restriction of our N\'eron model ${N}_{g,n}$. 
%
%
%
%
%\begin{definition}\label{def:Mle}
%Let $n$, $g \ge 0$ such that $2g-2 + n >0$, and let $\beta\colon \Mtil{g,n} \ra \Mbar{g,n}$ be the universal N\'eron-model-admitting morphism.  Define $\beta^{\le 1} \colon \Mtil{g,n}^{\le 1} \ra \Mbar{g,n}$ to be the smallest open substack of $\Mtil{g,n}$ such that the map $\beta^{\le 1}$ satisfies the valuative criterion for properness for all morphisms from traits $T$ to $\Mbar{g,n}$ such that:
%\begin{enumerate}
%\item the generic point of $T$ maps into $\M{g,n}$;
%\item  the corresponding stable curve $X/T$ is regular. 
%\end{enumerate}
%\end{definition}
%
%%As such, any trait $T \ra \Mbar{g}$ for which Caporaso's conditions are satisfied will lift to a morphism $\Mtil{g}^{\le 1}$. \chck{too vague}
%

The next proposition shows that the restriction of the N\'eron model of the universal jacobian to the open substack $\Mtil{g}^{\le 1}$ is given by the pullback of Caporaso's construction. 

%Add something here about N\'eron models of torsors. Then can remove the $n \ge 1$ below, and just take $n=0$ to fit in with Caporaso. 

\begin{proposition}
Let $g \ge 3$ and $d$ be integers such that $\on{gcd}(d-g+1,2g-2)=1$. Write $\tilde{\ca{P}}$ for the pullback of $\ca{P}_{d,g}$ to $\Mtil{g}^{\le 1}$, and write ${N}_{g}$ for the N\'eron model of $J_{g}$ over $\Mtil{g}^{\le 1}$. Then the canonical map $h\colon \tilde{\ca{P}} \ra {N}_{g}$ given by the N\'eron mapping property is an isomorphism. 
%Let $g \ge 3$ and $d$ be integers such that $\on{gcd}(d-g+1,2g-2)=1$. Write $\ca{P}$ for the pullback of $\ca{P}_{d,g}$ to $\Mbar{g,n}$ via the forgetful map, and view $\ca{P}$ as a model of the universal jacobian $J_{g,n}$ over $\M{g,n}$ via the first of the $n$ marked sections. Write $\tilde{\ca{P}}$ for the pullback of $\ca{P}$ to $\Mtil{g,n}^{\le 1}$, and write ${N}_{g,n}$ for the N\'eron model of $J_{g,n}$ over $\Mtil{g,n}^{\le 1}$. Then the canonical map $h\colon \tilde{\ca{P}} \ra {N}_{g,n}$ given by the N\'eron mapping property is an isomorphism. 
\end{proposition}
\begin{proof}
It is enough to check the map $h$ is an isomorphism on every geometric fibre over $\Mtil{g}^{\le 1}$. Let $p$ be a geometric point of $\Mtil{g}^{\le 1}$. Then there exists a trait $T$ with geometric closed point $t$, and a morphism $g\colon T \ra \Mbar{g}$ such that 
\begin{enumerate}
\item the given stable curve $X \ra T$ is regular;
\item the map $g$ factors via a map $\tilde{g} \colon T \ra \Mtil{g}^{\le 1}$;
\item this factorisation $\tilde{g}$ maps $t$ to $p$. 
\end{enumerate}

Now since $X$ is regular, we find that $\tilde{g}^*{N}_{g}$ is the N\'eron model of the jacobian of the generic fibre of $X \ra T$, and the same holds for $\tilde{g}^*\tilde{\ca{P}} = g^*\ca{P}$. In particular, this shows that the fibres of ${N}_{g}$ and of $\tilde{\ca{P}}$ over $p$ are isomorphic. Moreover, the given map between them is an isomorphism; this is true because it is so over the generic point of $T$ (apply the uniqueness part to the N\'eron mapping property). 
\end{proof}

In particular, this shows that, after pullback along the morphism $\Mtil{g}^{\le 1} \ra \Mbar{g}$, the stack $\ca{P}_{d,g}$ admits a natural torsor structure extending that over $\M{g}$.

\bibliographystyle{alpha} %amsplain}
\bibliography{../../prebib.bib}

\end{document}